\documentclass{amsart}
\usepackage[utf8]{inputenc}

\usepackage{amsfonts}
\usepackage{amsmath}
\usepackage{amssymb}
\usepackage{amsthm}
\usepackage{dutchcal}
\PassOptionsToPackage{hyphens}{url}\usepackage[colorlinks=true,citecolor={blue},linkcolor = purple]{hyperref}
\usepackage{enumerate}
\usepackage{enumitem}
\usepackage{extarrows}
\usepackage{mathdots}
\usepackage{mathrsfs}
\usepackage{tikz}
\usepackage{tikz-cd}
\usepackage{todonotes}
\usepackage{verbatim}
\usepackage{mathtools} 

\usepackage{wasysym} 

\usepackage{comment} 

\newcommand{\scr}[1]{\mathscr{#1}}

\DeclareFontFamily{U}{wncy}{}
    \DeclareFontShape{U}{wncy}{m}{n}{<->wncyr10}{}
    \DeclareSymbolFont{mcy}{U}{wncy}{m}{n}
    \DeclareMathSymbol{\Sha}{\mathord}{mcy}{"58}
    \DeclareMathSymbol{\Den}{\mathord}{mcy}{"44}
    \DeclareMathSymbol{\Num}{\mathord}{mcy}{"4E}

\theoremstyle{definition}
\newtheorem{theorem}{Theorem}[section]
\newtheorem{lemma}[theorem]{Lemma}

\newtheorem{corollary}[theorem]{Corollary}
\newtheorem{proposition}[theorem]{Proposition}

\newtheorem{definition}[theorem]{Definition}
\newtheorem{example}[theorem]{Example}
\newtheorem{remark}[theorem]{Remark}
\newtheorem{question}[theorem]{Question}

\newtheorem{situation}[theorem]{Situation}

\newcommand{\colim}{\text{colim}}

\newcommand{\OO}{\mathcal{O}}

\newcommand{\AAut}{\mathrm{Aut}}

\newcommand{\Spec}{\mathrm{Spec}\,}

\newcommand{\ZZ}{\mathbb{Z}}

\newcommand{\Aff}{\mathbb{A}}

\newcommand{\Sch}[1]{(Sch/#1)}

\newcommand{\cal}[1]{\mathcal{#1}}

\renewcommand{\frak}[1]{\mathfrak{#1}}

\newcommand{\localindex}{{\mathcal{i}}}
\newcommand{\indexsheaf}{{\mathcal{I}}}
\newcommand{\classicalindex}{{\mathrm{Index}}}

\definecolor{sebgreen1}{rgb}{0.019,0.317,0.149}
\definecolor{sebgreen2}{rgb}{0.784,0.952,0.780}

\renewcommand{\epsilon}{\varepsilon}

\newcommand{\e}{\epsilon}

\renewcommand{\AA}{\mathbb{A}}

\newcommand{\CC}{\mathbb{C}}
\newcommand{\FF}{\mathbb{F}}
\newcommand{\GG}{\mathbb{G}}

\newcommand{\NN}{\mathbb{N}}
\newcommand{\PP}{\mathbb{P}}
\newcommand{\QQ}{\mathbb{Q}}

\counterwithin{figure}{section}

\newcommand{\stquot}[1]{{\left[ #1 \right]}}

\renewcommand{\tilde}[1]{\widetilde{#1}}

\usepackage{ifthen}
\newboolean{seriousversion}
\setboolean{seriousversion}{true}

\ifthenelse{\boolean{seriousversion}}{\excludecomment{joke}}{}

\renewcommand{\cong}{\simeq}

\newcommand{\id}{\mathrm{id}}

\newcommand{\action}{\:\rotatebox[origin=c]{-90}{$\circlearrowright$}\:}

\newcommand{\Hilb}{\text{Hilb}}
\newcommand{\Pic}{\text{Pic}}
\newcommand{\Hom}{\text{Hom}}
\newcommand{\HHom}{\underline{\text{Hom}}}
\newcommand{\Gen}{\mathcal{M}}
\newcommand{\NGen}{\mathcal{N}}

\newcommand{\Sym}[1]{\text{Sym}^{#1}}

\newcommand{\WR}{\cal R}

\usepackage{stmaryrd}
\newcommand{\adj}[1]{\llbracket #1 \rrbracket}

\newcommand{\GL}{\text{GL}}

\newcommand{\pb}{\ar[phantom, dr, very near start, "\ulcorner"]}

\newcommand{\jetsp}[1]{J_{#1}}

\newcommand{\Tr}{\operatorname{Tr}} 

\newcommand{\Ints}{\mathbb{Z}} 

\newcommand{\vct}[1]{\vec{#1}}

\newcommand{\disc}{\operatorname{Disc}}

\usepackage{authblk}
\usepackage[foot]{amsaddr}

\usepackage[
backend=biber,
style=alphabetic
]{biblatex}
\makeatletter
\@namedef{subjclassname@2020}{%
  \textup{2020} Mathematics Subject Classification}
\makeatother

\title[The Scheme of Monogenic Generators I: Representability]{The Scheme of Monogenic Generators I: Representability}
\author{Sarah Arpin$^1$, Sebastian Bozlee$^2$, Leo Herr$^3$, Hanson Smith$^4$}
\date{\today}

\email{sarah.arpin@colorado.edu}
\email{sebastian.bozlee@tufts.edu}
\email{herr@math.utah.edu}
\email{hanson.smith@uconn.edu}

\address{$^1$University of Colorado Boulder, $^2$Tufts University, $^3$University of Utah, $^4$University of Connecticut}

\begin{document}

\maketitle

\sloppy

\begin{abstract}
     This is the first in a series of two papers that study monogenicity of number rings from a moduli-theoretic perspective. Given an extension of algebras $B/A$, when is $B$ generated by a single element $\theta \in B$ over $A$? In this paper, we show there is a \textit{scheme} $\mathcal{M}_{B/A}$ parameterizing the choice of a generator $\theta \in B$, a ``moduli space'' of generators. This scheme relates naturally to Hilbert schemes and configuration spaces. We give explicit equations and ample examples.
\end{abstract}

\tableofcontents

\section{Introduction}\label{sec: Intro}

The theorem of the primitive element states that given a finite separable field extension $L/K$, there is an element $\theta$ of $L$ such that $L = K(\theta)$. This holds for any extension of number fields $L/K$. By contrast, the extension of their rings of integers $\Ints_L/\Ints_K $ may require up to $\lceil \log_2([L : K]) \rceil$ elements of $\Ints_L$ to generate $\Ints_L$ as a $\Ints_K$-algebra \cite{Pleasants}. 

\begin{question}\label{MotivatingQuestion}
Which extensions $\Ints_L / \Ints_K$ are generated by a single element over $\Ints_K$? More generally, which finite locally free algebras $B /A$ are generated by a single element over $A$? How many elements does it take to generate $B$ over $A$ otherwise? 
\end{question}

\begin{definition}
A finite locally free $A$-algebra $B$ is \emph{monogenic}\footnote{The literature often uses the phrase ``$L$ is monogenic over $K$" to mean $\Ints_L/\Ints_K$ is monogenic as above. We prefer ``$\Ints_L$ is monogenic over $\Ints_K$" in order to treat fields and more exotic rings uniformly. If $B$ is monogenic of degree $n$ over $A$ with monogenerator $\theta$, the elements $\{1, \theta, \theta^2, \dots, \theta^{n-1}\}$ are elsewhere referred to as a ``power $A$-integral basis.'' } if there is an element $\theta \in B$ such that $B = A[\theta]$. The element $\theta$
is called a \emph{monogenic generator} or \emph{monogenerator} of $B$ over $A$.

If there are elements $\theta_1, \ldots, \theta_k \in B$ such that $B = A[\theta_1, \ldots, \theta_k]$, then $B/A$ is \emph{$k$-genic} and $(\theta_1, \ldots, \theta_k)$ is a \emph{generating $k$-tuple}.
\end{definition}

\medskip

This paper is motivated by the observation that monogenicity\footnote{`Monogeneity' is also common in the literature.} of an algebra can be restated geometrically:

\begin{remark} \label{rmk:monogenerator_equals_closed_embedding}
Let $A \subseteq B$ be an inclusion of rings. A choice of element $\theta \in B$ is equivalent to a commutative triangle
\begin{equation}\label{eqn:commtrianglemaptoA1}
\begin{tikzcd}
\Spec B \ar[rr, dashed, "s_\theta"] \ar[dr]         &       &\Aff^1_A, \ar[dl]       \\
        &\Spec A
\end{tikzcd}
\end{equation}
where $s_\theta$ is the map induced by the $A$-algebra homomorphism $A[t] \to B$ taking $t$ to $\theta$. 
\begin{center}
\emph{The element $\theta$ is a monogenerator if and only if the map $s_\theta$ is a closed immersion.}
\end{center}

Likewise, a $k$-tuple $\vec{\theta} = (\theta_1, \dots, \theta_k) \in B^k$ determines a corresponding map $s_{\vec{\theta}}: \Spec B \to \Aff^k_A$ induced by the $A$-algebra homomorphism $A[t_1, \ldots, t_k] \to B$ taking $t_i \mapsto \theta_i$. The tuple $(\theta_1, \ldots, \theta_k)$ generates $B$ over $A$ if and only if $s_{\vec{\theta}}$ is a closed immersion.
\end{remark}

We ask if there is a scheme that represents such commutative triangles as in moduli theory \cite[Chapter 1]{fgaexplained}. We prove that for any Noetherian ring $A$ and finite locally-free $A$-algebra $B$, there is a representing scheme $\Gen_{B/A}$ over $\Spec A$ called the \emph{scheme of monogenerators}. There is an analogous \emph{scheme of generating $k$-tuples} or \emph{scheme of $k$-generators} denoted by $\Gen_{k, B/A}$. When the extension $B/A$ is implied, we simply write $\Gen$ or $\Gen_k$ as appropriate. 

\begin{theorem}
Let $B / A$ be a finite locally free extension of Noetherian rings. There is an \emph{affine} $A$-scheme $\Gen_{B/A}$ and finite type quasiaffine $A$-schemes $\Gen_{k, B/A}$, for $k \geq 1$, with natural bijections
\[
  \Hom_{\Sch{A}}(\Spec A, \Gen_{B/A}) \cong \left\{ \theta \in B \mid \theta \text{ is a monogenerator for } B \text{ over }A \right\}
\]
and
\[
  \Hom_{\Sch{A}}(\Spec A, \Gen_{k, B/A}) \cong \left\{ \vec{\theta} \in B^k \mid \vec{\theta} \text{ is a generating $k$-tuple for }B \text{ over }A \right\}.
\]
\end{theorem}
The problem of finding generating $k$-tuples of $\Ints_L / \Ints_K$ is thereby identified with that of finding $\Ints_K$-points of the schemes $\Gen_{k, \Ints_L/\Ints_K}$. The functors $\Hom(-, \Gen_k)$ automatically form sheaves in the fpqc topology (and therefore also in coarser topologies). This permits monogenicity to be studied locally, a perspective that we will pursue further in the second paper in this series.

The extensions of rings $B/A$ under consideration are much more general than extensions of number rings, and there is no difficulty in extending the construction of $\Gen_{B/A}$ to maps of schemes $S' \to S$ locally of the form $\Spec B \to \Spec A$ with $B$ finite locally free over $A$. For example, a finite map of algebraic curves $C \to D$ is of this form. We are therefore invited to view monogenicity in other contexts as related and analogous to monogenicity of number rings. 

Some of these other extensions of rings are nevertheless functorially related to extensions of number rings. For example, for every extension of number rings $\Ints_L / \Ints_K$, there is a map $\Spec \CC \to \Spec \Ints_K$. Pulling back $\Spec \Ints_L \to \Spec \Ints_K$ along this map yields $\Spec \CC^n \to \Spec \CC$, a trivial cover of a point. Pulling back $\Gen_{\Ints_L/\Ints_K}$ along the same map yields $\Gen_{\CC^n/\CC}$. This part of the monogenicity space is already interesting:

\begin{example}\label{ex:confspsintro}
Let $A = \CC$ and $B = \CC^n$. The complex points of the monogenicity space are naturally in bijection with the points of the configuration space of $n$ distinct points in $\CC$:
\[\Gen_{1, B/A}(\CC) \simeq \text{Conf}_n(\CC) \coloneqq  \{(x_1, \dots, x_n) \in \CC^n \, | \, x_i \neq x_j \text{ for }i \neq j\}.\]
\end{example}

Monogenicity therefore generalizes configuration spaces by conceiving of $B/A$ as ``families of points'' to be configured in $\Aff^1$.

\subsection{Outline of the paper}

Section \ref{s:gen} defines our main object of study $\Gen_{S'/S}$ for a finite locally free map $S' \to S$. This scheme parametrizes choices of monogenerators for the algebra extension $\OO_{S'}/\OO_S$. We offer basic properties, functoriality, examples, and a relation with the classical Hilbert scheme and the work of Poonen \cite{poonenmodspoffiniteflat}. 

In Section \ref{sec:indexformconstgen}, we obtain explicit affine charts for $\Gen_{S'/S}$ using local index forms and deduce affineness for $\Gen_{B/A}$. The index form tells whether a given section $\theta \in \OO_{S'}$ is a monogenerator for $ \OO_{S'}/\OO_S$ or not. We generalize to $k$-generators. 

Section \ref{s:examples} gives a variety of concrete examples of the scheme of monogenerators $\Gen$, including: separable and inseparable field extensions, a variety of orders in number fields including Dedekind's non-monogenic cubic, jet spaces, and completions. These explicit equations put into practice classical theory in addition to the theory we have built. We encourage the reader to consult these examples to complement the earlier sections.

\subsection{Summary of previous results}\label{previouswork}

The question of which rings of integers are monogenic was posed to the London Mathematical Society in the 1960's by Helmut Hasse. Hence the study of monogenicity is sometimes known as \textit{Hasse's problem}. For an in-depth look at monogenicity with a focus on algorithms for solving index form equations, see Ga\'al's book \cite{GaalsBook}.
Evertse and Gy\H{o}ry's book \cite{EvertseGyoryBook} provides background with a special focus on the relevant Diophantine equations. For another bibliography of monogenicity, see Narkiewicz's texts \cite[pages 79-81]{Nark} and \cite[pages 75-77]{AlgebraicStory}. 

The prototypical examples of number rings are monogenic over $\ZZ$, such as quadratic and cyclotomic rings of integers.
Dedekind \cite{Dedekind} produced the first example of a non-monogenic number ring (see Example~\ref{ex:Dedekind}). Dedekind used the splitting of the prime $2$ to show that the field obtained by adjoining a root of $x^3-x^2-2x-8$ to $\QQ$ is not monogenic over $\ZZ$. Hensel \cite{Hensel1894} showed that local obstructions to monogenicity come from primes whose splitting cannot be accommodated by the local factorization of polynomials in the sense of Dedekind-Kummer factorization. See \cite[Proposition III.12]{LocalFields} and also \cite{Pleasants} which is discussed briefly below. For very recent English translations of the original pioneering works, consult \cite{GWDedekind} and \cite{GWHensel}.

Global obstructions also preclude monogenicity. For a number field $L/\QQ$, the \emph{field index} is the greatest common divisor $\gcd_{\alpha\in \Ints_L} \left[\Ints_L:\ZZ[\alpha]\right]$. A number field $L$ can have field index 1 and $\ZZ_L$ may still not be monogenic, for example the ring of integers of $\QQ(\sqrt[3]{25\cdot7})$; see \cite[page 65]{Nark} or Example \ref{ex:2genicOverZ}. Define the \emph{minimal index} to be $\min_{\alpha\in \Ints_L} \left[\Ints_L:\ZZ[\alpha]\right]$. The monogenicity of $\ZZ_L/\ZZ$ is equivalent to having minimal index equal to 1. An early result of Hall \cite{Hall} shows that there exist cubic fields with arbitrarily large minimal indices. In \cite{SpearmanYangYoo}, this is generalized to show that every cube-free integer occurs as the minimal index of infinitely many radical cubic fields.

The monogenicity of a given extension of $\ZZ$ is encoded by a Diophantine equation called the \textit{index form equation}. 
Gy\H{o}ry made the initial breakthrough regarding the resolution of index form equations and related equations in the series of papers \cite{GyoryI}, \cite{GyoryII}, \cite{GyoryIII}, \cite{GyoryIV}, and \cite{GyoryV}. These papers investigate monogenicity and prove effective finiteness results for affine inequivalent monogenerators in a variety of number theoretic contexts. 
For inequivalent monogenic generators one should also consult \cite{EvertseGyory85}, \cite{MultiplyMono}, and the survey \cite{MonoSurvey}. 
Specializing families of polynomials to obtain monogenic extensions is investigated in \cite{Konig}. 
In large part due to the group in Debrecen, there is a vast literature involving relative monogenicity: \cite{GyorySeminar}, \cite{GyoryCrelle}, 
\cite{CubicRelative}, \cite{QuarticRelative}, \cite{QuartQuadRelative}, \cite{GaalRemeteSzabo}, \cite{GaalRemete0}, and \cite{GaalRemeteMonoRelative}.

Pleasants \cite{Pleasants} bounds the number of generators needed for a field of degree $n$ by $\lceil \log_2(n) \rceil$, with equality if 2 splits completely. This upper bound is a consequence of a precise description of exactly when an extension is locally $k$-genic in the sense that the completion at each prime is $k$-genic. For number rings, Pleasants answers the question of what the minimal positive integer $k>1$ is such that $\ZZ_L$ is $k$-genic. Global obstructions only appear in the case of monogenicity: for $k>1$, Pleasants shows that local $k$-geneity is equivalent to global $k$-geneity.

A related account is given in \cite[Chapter 11]{EvertseGyoryBook}, where it is shown that given an order $\mathfrak{O}$ of a finite \'etale $\QQ$-algebra, one can effectively compute the smallest $k$ such that $\mathfrak{O}$ is $k$-genic. In the spirit of the previous work of Gy\H{o}ry, one can also effectively compute the $k$ generators of $\mathfrak{O}$ over $\ZZ$.

Monogenicity has recently been viewed from the perspective of arithmetic statistics: Bhargava, Shankar, and Wang \cite{BSW} 
have shown that the density of monic, irreducible polynomials in $\ZZ[x]$ such that a root is a monogenerator is $\frac{6}{\pi^2}=\zeta(2)^{-1} \approx 60.79\%$. That is, about $61\%$ of monic, integer polynomials correspond to monogenerators. They also show the density of monic integer polynomials with square-free discriminants (a sufficient condition for a root to be a monogenerator) is 
\[\prod\limits_p \left(1-\dfrac{1}{p}+\dfrac{(p-1)^2}{p^2(p+1)}\right)\approx 35.82\%.\]
Thus these polynomials only account for slightly more than half of the polynomials yielding a monogenerator. Using elliptic curves, \cite{AlpogeBhargavaShnidman} shows that a positive proportion of cubic number fields are not monogenic despite having no local obstructions. More recently, the trio have undertaken a similar investigation for quartic fields \cite{ABSQuartic}. 
For quartic orders, Bhargava \cite{BhargavaQuarticOrders} also establishes a new upper bound on the number of essentially different monogenerators. In a pair of papers that investigate a variety of questions (\cite{SiadPartI} and \cite{SiadPartII}), Siad shows that monogenicity appears to increase the average amount of 2-torsion in the class group of number fields. In particular, monogenicity has a doubling effect on the average amount of 2-torsion in the class group of odd degree number fields. Previously,  \cite{BhargavaHankeShankar} had established this result in the case of cubic fields.

\subsection{Acknowledgements}

The third author would like to thank Gebhard Martin, the mathoverflow community for \cite{370972} and \cite{377840}, Tommaso de Fernex, Robert Hines, and Sam Molcho. Tommaso de Fernex looked over a draft and made helpful suggestions about jet spaces. The third author thanks the NSF for providing partial support by the RTG grant \#1840190. 

The fourth author would like to thank Henri Johnston and Tommy Hofmann for help with computing a particularly devious relative integral basis in Magma. All four authors would like to thank their graduate advisors Katherine E. Stange (first and fourth authors) and Jonathan Wise (second and third authors). This project grew out of the fourth author trying to explain his thesis to the second author in geometric terms.

For numerous computations throughout, we were very thankful to be able to employ Magma \cite{Magma} and SageMath \cite{Sage}. For a number of examples the \cite{lmfdb} was invaluable.


\section{The scheme of monogenic generators}\label{s:gen}

We now use Remark \ref{rmk:monogenerator_equals_closed_embedding} to construct the scheme of monogenic generators $\Gen_{S'/S}$, our geometric reinterpretation of the classical question of monogenicity.
Given any extension of number fields $L/K$, the map $S' \coloneqq \Spec \Ints_L \to S \coloneqq \Spec \Ints_K$ is finite locally free and $X = \Aff^k_{\Ints_K} \to \Spec \Ints_K$ is quasiprojective; these properties suffice for our purposes.
Leaving $X$ general permits analogues of monogenicity such as embeddings into $\PP^k$, which are more natural when $S' \to S$ is a map of proper varieties. We invite the reader to picture $S'$ as the ring of integers $\Spec \Ints_L$ in a number field, $S = \Spec \ZZ_K$, and $X$ as $\Aff^1$. We return to $X = \Aff^1$, $\Aff^k$ in Section \ref{sec:indexformconstgen}; we will see there that our approach recovers the well-known index form equations. 

\begin{situation}\label{sit:gensetup}
Let $\pi : S' \to S$ be a finite locally free morphism of constant degree $n \geq 1$ with $S$ locally noetherian. Consider a quasiprojective morphism $X \to S$ and write $X' \coloneqq  X \times_S S'$. 
\end{situation}

The constant degree assumption is for simplicity; the reader may remove it by working separately on each connected component of $S$.
In the sequel paper, we will occasionally allow $S$ to be an algebraic stack,
though the morphism $S' \to S$ will always be representable. Write $\Sch{S}$ for the category of $S$-schemes.
We now define $\Gen_{X,S'/S}$ by describing the functor of maps into it and then showing it is representable.

\begin{definition}\label{def:defnofmonogenweilrestn}
In Situation \ref{sit:gensetup}, consider the presheaf on $\Sch{S}$ which sends an $S$-scheme $T \to S$ to the set of morphisms $s$ fitting into the diagram:
\begin{equation}\label{eqn:weilrestnsections}
T \to S \mapsto 
\left\{\begin{tikzcd}
S'\times_S T \ar[rr, "s", dashed] \ar[dr]        &       &X \times_S T \ar[dl]       \\
        &T
\end{tikzcd}\right\}
\end{equation}
{\noindent}with restriction given by pullback.
Refer to this presheaf either as the \emph{relative hom presheaf} $\HHom_S(S', X')$ \cite[0D19]{sta} or the \emph{Weil Restriction} $\WR_{X', S'/S}$ \cite[\S 7.6]{neronmodels}. 

The \emph{sheaf of monogenerators} is the subpresheaf
\[
  \Gen_{X, S'/S} \subseteq \WR_{X', S'/S}
\]
whose sections are diagrams exactly as above, but with $s$ a closed immersion. 

We write $\Gen_{k, S'/S}, \WR_{k, S'/S}$ for $X = \Aff^k_S$,
$\Gen_{S'/S} = \Gen_{1, S'/S},$ and $\WR_{S'/S} = \WR_{1, S'/S}$. When $S' \to S$ or $X$ is understood, we may drop ``$S'/S$'' or ``$X$'' from the notation. If $S' = \Spec B$ and $S = \Spec A$ are affine, we
write $\Gen_{k, B/A}$ or $\Gen_{B/A}$ instead.
\end{definition}

Remark \ref{rmk:monogenerator_equals_closed_embedding} expresses the presheaf $\Gen_{k, S'/S}$ as
\[
  \Gen_{k, S'/S}(T) = \left\{ (\theta_1, \ldots, \theta_k) \in \Gamma(T \times_S S', \OO_{T \times_S S'})^{\oplus k} \mid \OO_{T \times_S S'} = \OO_{T}[\theta_1, \ldots, \theta_k] \right\}
\]
An $S$-point of $\Gen_1$ is said to be a \emph{monogenerator} of $S'/S$ and we say $S'/S$ is \emph{monogenic} if such a point exists. This recovers the definition of monogenicity of algebras when $S$ is affine. These presheaves are representable: 

\begin{proposition}\label{prop:genisrep}
The presheaves $\Gen_X \subseteq \WR_{X', S'/S}$ are both representable by quasiprojective $S$-schemes and the inclusion is a quasicompact open immersion. If $f : X \to S$ is smooth, unramified, or \'etale, the same is true for $\WR_{X', S'/S} \to S$ and $\Gen_X \to S$. 
\end{proposition}

\begin{proof}
The Weil restriction $\WR_{X', S'/S}$ is representable by a scheme quasiprojective over $S$ \cite[Theorem 1.3, Proposition 2.10]{weilrestnpatmcfaddin}. If $X' \to S'$ is a finite-type affine morphism, the same is true for $\WR_{X', S'/S} \to S$ by locally applying \cite[Proposition 2.2(1),(2)]{weilrestnpatmcfaddin}. The inclusion $\Gen_X \subseteq \WR_{X', S'/S}$ is open by \cite[05XA]{sta} and automatically quasicompact because $\WR_{X', S'/S}$ is locally noetherian. The second statement is immediate. 
\end{proof}

\begin{corollary}
The presheaves $\Gen_X, \WR_{X', S'/S}$ are sheaves in the Zariski, Nisnevich, \'etale, fppf, and fpqc topologies on $\Sch{S}$. 
\end{corollary}

\begin{definition}\label{def:sheafmono}
Let $\tau$ be a subcanonical Grothendieck topology on schemes, for example the Zariski, Nisnevich, \'etale, fppf, or fpqc topologies.
We say that $S'/S$ is $\tau$-\emph{locally $k$-genic} if the sheaf $\Gen_{k, S'/S}$ is locally non-empty in the topology $\tau$.
I.e., there is a $\tau$-cover $\{ U_i \to S \}_{i \in I}$ of $S$ such that $\Gen_{k, S'/S}(U_i)$ is non-empty for all $i \in I$.
By default, we use the \'etale topology.
\end{definition}

A $k$-genic extension $S'/S$ is $\tau$-locally $k$-genic. If $\tau_1$ is a finer topology than $\tau_2$, then $\tau_2$-locally monogenic implies $\tau_1$-locally monogenic.

\begin{remark}
We pose a related moduli problem $\scr F$ in Section~\ref{ss:finflat+gen} that parameterizes a \emph{choice} of finite flat map $S' \to S$ together with a monogenerator. It is also representable by a scheme. The mere choice of a finite flat map $S' \to S$ is representable by an algebraic stack, as shown in \cite{poonenmodspoffiniteflat} and recalled in Section \ref{ssec:relation_to_Hilb}. 
\end{remark}

Our main example of $S' \to S$ comes from rings of integers in number fields $\Ints_L/\Ints_K$, but here is another:

\begin{example}\label{ex:jetsps}
Let $S = \Spec \ZZ$ and $S'_n = \Spec \ZZ[\e]/\e^n$. The Weil Restriction $\WR_{X', S'_n/S}$ is better known as the \textit{jet space} $\jetsp{X, n-1}$ \cite{vojtajets}. For any ring $A$, $(n-1)$-jets are maps
\[\Spec A[\e]/\e^n \to X.\]
Jet spaces are usually considered over a field $k$ by base changing from $S$. The monogenicity space $\Gen_{S'_n/S, X} \subseteq \jetsp{X, n-1}$ parametrizes \emph{embedded} $(n-1)$-jets, whose map $\Spec A[\e]/\e^n \subseteq X$ is a closed embedding. If $n=2$, $\jetsp{X, 1}$ is the Zariski tangent bundle and $\Gen_{X, S'_2/S}$ is the complement of the zero section. 

The truncation maps $\jetsp{X, n} \to \jetsp{X, n-1}$ restrict to maps $\Gen_{X, S'_{n+1}/S} \to \Gen_{X, S'_n/S}$. The inverse limit $\lim_n \jetsp{X, n}$ 
sends rings $A$ to maps called \textit{arcs}
\[\Spec A \adj{t} \to X\]
according to \cite[Remark 4.6, Theorem 4.1]{bhatttannakaalgn}. Under this identification, the limit $\lim_n \Gen_{X, S'_n/S}$ parametrizes those arcs that are closed immersions into $X \times \Spec A$. 

Compare embedded $(n-1)$-jets to ``regular'' ones. An $(n-1)$-jet $f : \Spec k[\e]/\e^{n} \to X$ over a field $k$ is called \emph{regular} \cite[\S 5]{demaillyhyperbolicjetspace} if $f'(0) \neq 0$. I.e., the truncation of higher order terms $\overline{f} : \Spec k[\e]/\e^2 \to X$ is a closed immersion. Regular $(n-1)$-jets are precisely the pullback $\Gen_{X, S'_1/S} \times_{\jetsp{1, X}} \jetsp{X, n}$. 
\end{example}

Extensions of number rings are generically \'etale, with a divisor of ramification. The finite flat map $S'_n \to S$ in jet spaces is the opposite, ramified everywhere.

\begin{remark}[Steinitz Classes]

We have assumed that $\pi : S' \to S$ is finite locally free of rank $n$, so $\pi_*\OO_{S'}$ is a locally free $\OO_S$-module of rank $n$. By taking an $n$th
exterior power, one obtains a locally free $\OO_S$-module 
\[
  \det \pi_*\OO_{S'} \coloneqq   \bigwedge^n \pi_*\OO_{S'}.
\]
of rank 1 \cite[Chapter II, Exercise 6.11]{hartshorne}.
The \emph{Steinitz class} of $\pi : S' \to S$ is the isomorphism class of $\det \pi_*\OO_{S'}$ in $\Pic(S)$.
\end{remark}

The Weil Restriction $\WR_{k, S'/S}$ is precisely the rank $kn$ vector bundle with sheaf of sections $\pi_*\OO_{S'}^k$. To see this, recall that the set of $T$-points of the $S$-scheme $\WR_{k, S'/S}$ is by definition the set of $T$-morphisms $\left\{ S'\times_S T \to \Aff^k_T \right\}$.
By the universal property of $\Aff^k_T$, such morphisms are in bijection with $k$-tuples of elements of $\Gamma(T, \OO_{S' \times_S T})$.
It follows that $\Hom_{S}(-, \WR_{k, S'/S}) \cong (\pi_*\OO_{S'})^k$ as quasicoherent sheaves on $S$. When $\pi_\ast \OO_{S'} \simeq \bigoplus^n \OO_S \cdot e_i$ is trivial, so is $\WR_{S'/S}$. 

\begin{example}
The Steinitz class of the jet space of $\Aff^k$ is the trivial vector bundle:
$J_{n, \Aff^k} = \Aff^{k(n+1)}$
\cite[Corollary 5.2]{vojtajets}
\end{example}

\begin{remark}
Although $\Aff^1$ and hence $\WR_{S'/S}$ are ring objects, $\Gen_1$ is neither closed under addition nor multiplication.
For addition, note that if $\theta \in \OO_{S'}$ is a generator, then so is $-\theta$, but $\theta + (-\theta) = 0$ is not for $n \neq 1$. 
\end{remark}

\begin{lemma} \label{lem:minpoly}
Let $\theta \in \Gamma(S', \OO_{S'})$ be any element. There is
a \emph{canonical} monic polynomial $m_\theta(t) \in \Gamma(S, \OO_S)[t]$ of degree $n$ such
that $m_\theta(\theta) = 0$. 
\end{lemma}

\begin{proof}
We begin by constructing $m_\theta(t)$ locally, following \cite[Proposition 2.4]{AtiyahMacdonald}.

Assume first that $\pi : S' \to S$ is such that $\pi_*\OO_{S'} \cong \OO_S^{\oplus n}$ as an $\OO_S$-module.
Choose a $\Gamma(S, \OO_S)$-basis $\{ x_1, \ldots, x_n \}$ of $\Gamma(S', \OO_{S'})$. For each $i = 1, \ldots, n$, write $\theta x_i = \sum_{j = 1}^n a_{ij} x_j$ where $a_{ij} \in \Gamma(S', \OO_{S'})$.
Now we let
\[
  m_\theta(t) = \det(\delta_{ij}t - a_{ij}).
\]
As in the proof of \cite[Proposition 2.4]{AtiyahMacdonald}, $m_\theta(t)$
has coefficients in $\Gamma(S, \mathscr{O}_S)$, is monic of degree $n$, and $m_\theta(\theta) = 0$. Moreover,
$m_\theta(t)$ does not depend on the basis chosen since $m_\theta(t)$ is
computed by a determinant.

Now for general $\pi: S' \to S$, choose an open cover $\{U_i\}$ of $S$
on which $\pi_*\OO_{S'}$ trivializes. On each open set, the construction of the previous paragraph yields a monic
polynomial $m_i(t) \in \OO_S(U_i)[t]$ of degree $n$ vanishing on $\theta|_{U_i}$.
Since the construction of the polynomials commutes with restriction and
is independent of choice of basis, we have
\[
  m_i(t)|_{U_i \cap U_j} = m_j(t)|_{U_i \cap U_j}.
\]
We conclude by the sheaf property.
\end{proof}

\begin{remark}\label{rmk:minpolymap}
Lemma \ref{lem:minpoly} defines a map $m : \WR_{S'/S} \to \Aff^n_S$ sending an element $\theta \in \Gamma(S', \OO_{S'})$ to the coefficients $(b_{n-1}, \dots, b_0)$ of the universal canonical monic minimal polynomial 
\[m_\theta(t) = t^n + b_{n-1}t^{n-1} + \cdots + b_0.\] 
The preimage of a point of $\Aff^n_S$ is the set of roots of the corresponding polynomial in $\OO_{S'}$.

In the case that $S' \to S$ is a trivial $n$-sheeted cover, i.e. $S' = S \sqcup \cdots \sqcup S$, we may trivialize the vector bundle $\WR_{S'/S} \cong \AA^n_S$ using the standard basis $\{ e_1 = (1, 0, \cdots, 0), e_2 = (0, 1, \cdots, 0), \ldots \}$ of $\pi_*\OO_S \cong \OO_S \times \cdots \OO_S$. It is easy to compute that $m_\theta(t) = \prod_{i = 1}^n (t - x_i)$ with respect to this basis, so that the coefficients $b_i$ are the elementary symmetric polynomials in the $x_i$. It follows that the map $m : \WR_{S'/S} \to \Aff^n_S$
is the coarse quotient space for the natural action of the symmetric group $\Sigma_n$ on the $\{e_i\}$-coordinates of $\WR_{S'/S}$.

If $S' \to S$ is \'etale, $S' \to S$ is \'etale locally a trivial cover as above. However, the $\Sigma_n$-action need not globalize. For example, consider the $3$-power map $\stquot{3} : \GG_m \to \GG_m$ for $S' \to S$. We will consider the situation of $S' \to S$ \'etale in more depth in the second paper of this series.
\end{remark}

\subsection{Functoriality of $\Gen_X$}

We recall that the Weil restriction is functorial in the following sense. Given a commutative square
\[
  \begin{tikzcd}
    T' \ar[r] \ar[d]        &S' \ar[d]        \\
    T \ar[r]         &S
  \end{tikzcd}
\]
and an $S$-scheme $X$ (but not assuming that these maps lie in Situation \ref{sit:gensetup}), there is a functorially associated map of $T$-sheaves
\[
  \WR_{X, S'/S} \times_S T \to \WR_{X_T, T'/T},
\]
where $X_T \coloneqq X \times_S T$. Let $f : T' \to S' \times_S T$ be the map induced by the universal property of pullback.

The construction is as follows. Let $U \to T$ be a $T$-scheme and let $U' = S' \times_S U$. A $(U \to T)$-point of $\WR_{X, S'/S} \times_S T$ then consists of
a morphism $s : U' \to X \times_S U$ over $U$.  Let $f|_U: T' \times_T U \to U'$ be the pullback of $f$ along the map $U' \to S' \times_S T$:
\[
  \begin{tikzcd}
    T' \times_T U \ar[r] \ar[d, "f|_U"'] \pb & T' \ar[rd] \ar[d, "f"] \\
    U' \ar[r] \ar[d] \pb & S' \times_S T \pb \ar[r] \ar[d] & S' \ar[d] \\
    U \ar[r] & T \ar[r] & S
  \end{tikzcd}
\]
Then the natural map $\WR_{X, S'/S} \times_S T \to \WR_{X_T, T'/T}$ on $(U \to T)$-points is given by
\[
  (U' \overset{s}{\longrightarrow} X \times_S U) \, \, \mapsto \, \, (T' \times_T U \overset{f|_U}{\longrightarrow} U' \overset{s}{\longrightarrow} X \times_S U \cong X_T \times_T U).
\]

Assume now that $S' \to S$ and $X$ are as in Situation \ref{sit:gensetup}. Observe that if $f$ is a closed immersion, then composition with $f|_U$ preserves closed immersions and we obtain a natural map
\[
  \Gen_{X, S'/S} \times_S T \to \Gen_{X_T, T'/T}.
\]
by restriction.

If $f$ is not a closed immersion, then in general composition with $f|_U$ will
not preserve closed immersions, so $\Gen_{X, S'/S}$ fails to be functorial in an obvious way for general commutative squares. Indeed, $\Gen_{X, S'/S}$ cannot be functorial for general commutative squares. For example, if $S = \Spec A$, $T \to S \leftarrow S'$ are identity maps and $T' \to T$ is non-monogenic then the global sections of $\Gen_{S'/S} \times_S T$ are $A$, while $\Gen_{T'/T}$ has no global sections, so there does not even exist a map $\Gen_{S'/S} \times_S T \to \Gen_{T'/T}$.

We highlight some special cases in which $f$ is a closed immersion:
\begin{enumerate}
  \item (Base change in $S$) If
  \[
    \begin{tikzcd}
      T' \ar[r] \ar[d]        &S' \ar[d]        \\
      T \ar[r]         &S
    \end{tikzcd}
  \]
  is cartesian, then $f : T' \to S' \times_S T$ is an isomorphism, so the natural map
  \[
    \Gen_{X, S'/S} \times_S T \to \Gen_{X_T, T'/T}
  \]
  is an isomorphism. We will use this later to compute $\Gen_{S'/S}$ Zariski locally
  on $S$ by letting $T \to S$ vary over sufficiently small open affines of $S$.

  \item (Functoriality in towers) If $S'' \to S'$ is finite locally free of constant rank, consider the commutative square
  \[
    \begin{tikzcd}
      S'' \ar[r, "\id_{S''}"] \ar[d]        &S'' \ar[d]        \\
      S' \ar[r]         &S.
    \end{tikzcd}
  \]
  The natural map $f : S'' \to S'' \times_S S'$ is proper since $S'' \to S'$ is proper and $S'' \times_S S' \to S'$ is separated. It is also a monomorphism since $S'' \overset{f}{\to} S'' \times_S S' \to S'' = \id_{S''}$
  is a monomorphism. Therefore $f$ is a closed immersion\cite[Tag 04XV]{sta}.
  
  We therefore have a natural morphism
  \[
    \Gen_{X, S''/S} \times S' \to \Gen_{X_{S'}, S''/S'}.
  \]
  Taking global sections with $X = \Aff^k_S$, it follows that if $S'' \to S$ is $k$-genic, then $S'' \to S'$ is $k$-genic. In particular, if $M/L/K$ is a tower of number fields and $\Ints_M / \Ints_K$ is $k$-genic, then $\Ints_M / \Ints_L$ is $k$-genic.
More prosaically, if $A \subseteq B \subseteq C$ and $C = A[\theta_1, \ldots, \theta_k]$, then $C = B[\theta_1, \ldots, \theta_k]$.
\end{enumerate}

Weil restrictions are also functorial in $X$, in the following sense. Let $S' \to S$ be
any morphism and let $g : X \to X'$ be a morphism of $S$-schemes. If $U \to S$ is a morphism, write $g|_U$ for the pulled-back map $X \times_S U \to X' \times_S U$. Then there is a map
\[
  \WR_{X, S'/S} \to \WR_{X', S'/S}
\]
given on $(U \to S)$-points by sending $s : S' \times_S U \to X \times_S U$ to $g|_U \circ s$. If $g$ is a closed immersion, then composition with $g|_U$ preserves closed immersions,
inducing a map
\[
  \Gen_{X, S'/S} \to \Gen_{X', S'/S}.
\]

\begin{remark}\label{rmk:disjunionprodweilrest}
If $S' = \bigsqcup S'_i$ is a finite disjoint union of finite locally free maps $S'_i \to S$, the pullback $T \times_S S'$ is the disjoint union $\bigsqcup T \times_S S'_i$. Write $X'_i = X' \times_S S'_i$, it follows from the universal property of coproducts and the above that the Weil Restriction decomposes as
\[
  \WR_{X', S'/S} = \prod_i \WR_{X'_i, S'_i/S}.
\]

The monogenicity space is \textit{not} the product $\Gen_{S'/S, X} \neq \prod \Gen_{S'_i/S, X}$. Rather, we claim a map
\[
  \bigsqcup S'_i \to X
\]
is a closed immersion if and only if each map
\[
  S'_i \to X
\]
is a closed immersion and the closed immersions are disjoint: 
\[
  S'_i \times_X S'_j = \varnothing
\]
for all $i \neq j$. 
To see the claim, we may check affine locally on $X$, where it reduces to the statement that
$A \to \prod_{i} B_i$ is surjective if and only if $A \to B_i$ is surjective for each $i$ and $B_i \otimes_A B_j = 0$
whenever $i \neq j$. This follows quickly in turn from the Chinese remainder theorem.


\end{remark}

\begin{remark}\label{rmk:pbmonnormnofmon}
Not only is $\Gen_{1, S'/S}$ functorial in $S$, but we show its normalization and reduction can be performed on $S$.

If $X \to S$ is smooth and $T \to S$ is the normalization of $S$, one uses \cite[03GV]{sta} and properties of $\Gen_{X, S'/S}$ in Proposition \ref{prop:genisrep} to see that $\Gen_{X_T, T'/T} \to \Gen_{X, S'/S}$ is also normalization. 


If a map $Y \to Z$ is smooth, the square 
\[\begin{tikzcd}
Y_{red} \ar[r, hook] \ar[d]         &Y \ar[d]      \\
Z_{red} \ar[r, hook]         &Z.
\end{tikzcd}\]
is cartesian. We need only check $Z_{red} \times_Z Y$ is reduced using \cite[034E]{sta}, since a surjective closed immersion of a reduced scheme must be $Y_{red}$. For smooth $X \to S$ and $T \coloneqq   S_{red}$, the pullback $\Gen_{X, S'/S} \times_S S_{red} \simeq \Gen_{X_{S_{red}}, T'/S_{red}}$ is the reduction of $\Gen_{X, S'/S}$.

\end{remark}

\subsection{Relation to the Hilbert scheme}\label{ssec:relation_to_Hilb}

With this section we pause our development of $\Gen$ to relate our construction to other well-known objects. Though the rest of the paper does not use this section, it behooves us to situate our work in the existing literature.

At the possible cost of representability of $\Gen_{X,S'/S}$, let $X \to S$ be any morphism of schemes in this section. Recall the Hilbert scheme of points \cite[0B94]{sta}
\[\Hilb^n_{X/S}(T) \coloneqq   \big\{ \text{closed embeddings } Z \subseteq X \times_S T \, \, \big| \, \, Z \to T \text{ finite locally free deg. } n \big\}.\]

There is an algebraic moduli stack $\frak A_n$ of finite locally free maps of degree $n$ \cite[Definition 3.2]{poonenmodspoffiniteflat}, with universal finite flat map $\frak Z_n \to \frak A_n$. Any map $S' \to S$ in Situation \ref{sit:gensetup} is pulled back from $\frak Z_n \to \frak A_n$. We restrict to $S$-schemes without further mention: $\frak A_n = \frak A_n \times S$. 

Recall Poonen's description of $\frak A_n$: A finite locally free map $\pi : Z \to T$ is equivalent to the data of a finite locally free $\OO_T$-algebra $\cal Q$ given by $\pi_\ast \OO_Z$. Suppose for the sake of exposition that a locally free algebra $\cal Q$ has a global basis $\cal Q \simeq \OO_T^{\oplus n}$. 
The algebra structure is a multiplication map
\[\cal Q^{\otimes 2} \to \cal Q\]
that can be written as a matrix using the basis. Conditions of associativity and commutativity are polynomial on the entries of this matrix. We get an affine scheme of finite type $\frak B_n$ parametrizing matrices satisfying the polynomial conditions, or equivalently multiplication laws on globally free finite modules \cite[Proposition 1.1]{poonenmodspoffiniteflat}. Two different choices of global basis $\OO_T^{\oplus n} \simeq \cal Q \simeq \OO^{\oplus n}_T$ differ by an element of $\GL_n(\OO_T)$. Taking the stack quotient by this action $\GL_n \action \frak B_n$ erases the need for a global basis and gives $\frak A_n$.

There is a map $\Hilb^n_{X/S} \to \frak A_n$ sending a closed embedding $Z \subseteq X|_T$ to the finite flat map $Z \to T$. The fibers of this map are exactly monogenicity spaces:
\[\begin{tikzcd}
\Gen_{X, S'/S} \ar[r] \ar[d] \pb     &S    \ar[d, "S'/S"]    \\
\Hilb^n_{X/S} \ar[r]       &\frak A_n.
\end{tikzcd}\]
Conversely, the monogenicity space of the universal finite flat map $\frak Z_n \to \frak A$ is isomorphic to the Hilbert Scheme
\[\Gen_{X, \frak Z_n/\frak A_n} \simeq \Hilb^n_{X/S}\]
over $\frak A$. The space $\frak h_n(\Aff^k)$ of \cite[\S 4]{poonenmodspoffiniteflat} is $\Gen_k$ for the universal finite flat map to $\frak B_n$. 

Proposition \ref{prop:genisrep} shows $\Gen_{S'/S}(X) \to S$ is smooth for $X$ smooth, and likewise for unramified or \'etale. This means the map $\Hilb^n_{X/S} \to \frak A_n$ is smooth, unramified, or \'etale if $X \to S$ is.

Suppose $X \to S$ flat to identify the Chow variety of dimension 0, degree $n$ subvarieties of $X$ with $\Sym n X$ \cite{rydhfamiliesofcycles} and take $S$ equidimensional. The Hilbert-Chow morphism $\Hilb_{X/S}^{equi \: n} \to \Sym n X$ sending a finite, flat, equidimensional $Z \to S$ to the pushforward of its fundamental class $[Z]$ in Chow $A_\ast(X)$ restricts to $\Gen_{X, S'/S}$. When $S'/S$ is \'etale, we will see the restriction of the Hilbert-Chow morphism $\Gen_{X, S'/S} \to \Sym n X$ is an open embedding by hand in the sequel paper.

\begin{question}
Can known cohomology computations of $\Hilb^n_{X/S}$ offer obstructions to monogenicity under this relationship? 
\end{question}


\section{The local index form and construction of {$\Gen$}}\label{sec:indexformconstgen}


This section describes equations for the monogenicity space $\Gen_{1, S'/S}$ inside $\WR = \HHom_S(S',\Aff^1)$ by working with the universal homomorphism over $\WR$. In the classical case $\ZZ_L/\ZZ_K$ we recover the well-known \emph{index form equation}. Section \ref{subsec: firstexamples} gives examples, while \ref{ssec:genk_equations} generalizes the equations to $k$-geneity $\Gen_{k, S'/S}$.

\begin{remark}[Representable functors]
Recall that if a functor $F : \mathcal{C}^{op} \to \mathrm{Set}$ is represented by an object $X$, then there is an element $\xi$ of $F(X)$
corresponding to the identity morphism $X \overset{\id}{\to} X$, called the \emph{universal element} of $F$. The proof of the Yoneda lemma shows that for all objects $Y$ and elements $y \in F(Y)$, there is a morphism $f_y : Y \to X$ such that $y$ is obtained by applying $F(f_y)$
to $\xi$. 
\end{remark}

For a map $f : X \to Y$ of schemes, we write $f^\sharp : \OO_Y \to \OO_X|_Y$ for the map of sheaves and its kin.

\subsection{Explicit equations for the scheme $\Gen$}\label{ssec:gen1_equations}

The scheme $\WR = \WR_{S'/S} = \HHom_S(S', \Aff^1)$ is a ``moduli space'' of maps $S' \to \Aff^1$. For any $T \to S$, every morphism $S'\times_S T \to \Aff^1_T$ is pulled back along some $T \to \WR$ from the \textit{universal homomorphism}
\[\begin{tikzcd}
S' \times_S \WR \ar[dr] \ar[rr, "u"]   &       &\Aff^1_{\WR}. \ar[dl]       \\
        &\WR
\end{tikzcd}\]
We want explicit equations for $\Gen_{S'/S} \subseteq \WR$.

Let $t$ be the coordinate function on $\AA^1$. The map $u$ corresponds to an element $\theta = u^\sharp(t) \in \Gamma(\OO_{S' \times_S \WR})$. Let $m(t)$ be the polynomial of Lemma \ref{lem:minpoly} for $\theta$, i.e. $m(t)$ is a monic polynomial in $\Gamma(\OO_\WR)[t]$ of degree $n$ such that $m(\theta) = 0$.

\begin{definition}
We call the polynomial $m(t)$ the \emph{universal minimal polynomial} of $\theta$.
Let $V(m(t))$ be the closed subscheme of $\AA^1_\WR$ cut out by $m(t)$.
\end{definition}

The universal map $u$ factors through this closed subscheme: 
\begin{equation}
\label{eq:univ_factorization_mon}
\begin{tikzcd}
S' \times_S \WR \ar[r, "v"] \ar[dr, "\pi", swap]       &V(m(t)) \ar[d, "\tau"] \ar[r, hook]     &\Aff^1_\WR. \ar[dl]       \\
        &\WR
\end{tikzcd}
\end{equation}
Since $V(m(t)) \to \AA^1_{\WR}$ is a closed immersion, the locus in $\WR$ over which $u$ restricts
to a closed immersion agrees with the locus over which $v$ is a closed immersion. 

\begin{remark}
The map $V(m(t)) \to \WR$ is finite \emph{globally} free $\tau_\ast \OO_{V(m(t))} \simeq \bigoplus_{i=0}^{n-1} \OO_\WR \cdot t^i$. The map $v : S' \times_S \WR \to V(m(t))$ comes from a map
\[v^\sharp : \tau_\ast \OO_{V(m(t))} \to \pi_\ast \OO_{S' \times_S \WR}\]
of finite locally free $\OO_\WR$-modules. Locally, it is an $n \times n$ \emph{matrix}. The determinant of this matrix is a unit when it is full rank, i.e. when $v$ is a closed immersion. The $i$th column is $\theta^i$, written out in terms of the local basis of $\pi_\ast \OO_{S' \times_S \WR}$. We work this out explicitly to get equations for $\Gen \subseteq \WR$. 
\end{remark}

Remark \ref{rmk:pbmonnormnofmon} lets us find equations locally. 
Suppose $S' = \Spec B$ and $S = \Spec A$, where $B = \bigoplus_{i=1}^n A \cdot e_i$ is a finite free $A$-algebra of rank $n$ with basis $e_1, \ldots, e_n$. Let $I = \{ 1, \ldots, n \}$.
Write $t$ for the coordinate function of $\mathbb{A}^1$ and write $x_I$ as shorthand for $n$ variables $x_i$ indexed by $i \in I$.

The scheme $\WR = \HHom_S(S', \mathbb{A}^1)$ is the affine scheme $\Aff^n_S = \Spec A[x_I]$ and Diagram \eqref{eq:univ_factorization_mon} becomes
\[
\begin{tikzcd}
B[x_I]      &A[x_I, t]/(m(t)) \ar[l, "v^\sharp", swap]       &A[x_I, t] \ar[l]      \\
        &A[x_I] \ar[ur] \ar[ul, "\pi^\sharp"] \ar[u, "\tau^\sharp"]
\end{tikzcd}
\]
The $A[x_I]$-homomorphism $v^\sharp$ sends
\[
  t \mapsto \theta \coloneqq   x_1 e_1 + \cdots + x_n e_n.
\]

Note that $A[x_I, t]/(m(t))$ has an $A[x_I]$ basis given by the equivalence classes of $1, t, \dots, t^{n-1}$ and $B[x_I]$ has an $A[x_I]$-basis given by $e_1, \ldots, e_n$. With respect to these bases, $v^\sharp$ is represented by the \emph{matrix of coefficients}
\begin{equation}\label{eq:MatrixOfCoeffs}
  M(e_1, \ldots, e_n) = [ a_{ij} ]_{1\leq i,j\leq n} .
\end{equation}
where $a_{ij} \in A[x_I]$ are the unique coefficients such that $\theta^{j - 1} = \sum_{i = 1}^n a_{ij}e_i$ for each $j = 1, \ldots, n$.

\begin{example}\label{ex:etale} Let $S'\coloneqq  \Spec \CC^n \to S\coloneqq  \Spec \CC $ as in Example~\ref{ex:confspsintro}. Let $e_1,e_2,\dots,e_n$ be the standard basis vectors of $\CC^n$. The monogenerators of $S'$ over $S$ are precisely the closed immersions:

\begin{center}
\begin{tikzcd}
S' \ar[dr] \ar[rr, hook, dashed]      &       & \Aff^1_{S}\arrow[dl]\\
&S
\end{tikzcd}
\end{center}

Identify $\WR \simeq \Aff^n_S$. Let $t$ denote the coordinate of $\Aff^1_{\Aff^n_S}$, and let $x_1,x_2,...,x_n$ denote the coordinates of $\Aff^n_{S}$. The analogue of Diagram \eqref{eq:univ_factorization_mon} for this case is:

\begin{center}
\begin{tikzcd}
S' \times_S \Aff^n_S \ar[r] \ar[dr]
&V(m(t)) \ar[d] \ar[r, hook]         &\Aff^1_{\Aff^n_S} \ar[dl]      \\
        &\Aff^n_S
\end{tikzcd}    
\end{center}

We write down $m_{\theta}(t)$, as in Lemma~\ref{lem:minpoly}. The coordinate $t$ of $\Aff^1_{\Aff^n_S}$ maps to the universal element $\theta = x_1e_1 + x_2e_2 + \cdots + x_ne_n$. Since $e_i = (\delta_{ij})_{j=1}^n \in \CC^n$, we have $\theta e_{i} = x_ie_i$. Computing the minimal polynomial, $m_\theta(t) = \det(\delta_{ij}t - a_{ij}) = \prod_{i=1}^n(t -a_i)$.

Notice that $e_ie_j = \delta_{ij}e_i$. It follows that $\theta^i = \sum_{j = 0}^{n-1} x_j^ie_j$. Therefore $M(e_1, \ldots, e_n)$ is the Vandermonde matrix with $i$th row given by $\begin{bmatrix} 1 & x_i & x_i^2 & \cdots & x_i^{n - 1}\end{bmatrix}$.

\end{example}

\begin{definition}\label{def:local_index_form}
With notation as above, let $\localindex(e_1, \ldots, e_n) = \det(M(e_1, \ldots, e_n)) \in A[x_I]$. We call this element a \emph{local index form} for $S'$ over $S$. When the basis is clear from context, we may omit the basis elements from the notation.
\end{definition}

\begin{proposition} \label{prop:gen1_equations}
Suppose $S' \to S$ is finite free and $S$ is affine. With notation as above, $\Gen$ is the distinguished affine subscheme $D(\localindex(e_1, \ldots, e_n))$ inside $\WR \cong \Spec A[x_I]$.
\end{proposition}

\begin{proof}
By Proposition \ref{prop:genisrep}, $\Gen$ is an open subscheme of $\WR$. Therefore it suffices to check that $D(\localindex(e_1, \ldots, e_n))$ and $\Gen$ have the same points. Let $j : y \to \WR$ be the inclusion of a point with residue field $k(y)$. Then
\begin{align*}
  \text{$j$ factors through $\Gen$}\iff &\text{$j^*u$ is a closed immersion, where $u$ is the univ. hom.} \\
  \iff &\text{$j^*v$ is a closed immersion, for $v$ as in \eqref{eq:univ_factorization_mon}} \\
  \iff &\text{$v^\sharp \otimes_A k(y)$ is surjective} \\
  \iff &\text{$j^\sharp(M(e_1, \ldots, e_n))$ is full rank} \\
  \iff &\text{$j^\sharp(\localindex(e_1, \ldots, e_n))$ is nonzero.}
\end{align*}
This establishes the claim.
\end{proof}

\begin{remark}
If $B$ is a free $A$-algebra with basis $\{e_1, \ldots, e_n\}$, it follows that
$B$ is monogenic over $A$ if and only if there is a solution $(x_1, \ldots, x_n) \in A^n$ to one of the equations
\[
  \localindex(e_1,\ldots, e_n)(x_1,\ldots, x_n) = a
\]
as $a$ varies over the units of $A$. 
These are the well-known \emph{index form equations}. In the case that $A$ is a number ring there are only finitely many units, so only finitely many equations need be considered. This perspective
gives the set of \emph{global} monogenerators the flavor of a closed subscheme of $\WR$ even though $\Gen_1$ is an open subscheme.
\end{remark}

\begin{corollary} \label{thm:mon_is_affine}
The map $\Gen_{1, S'/S} \to S$ classifying \emph{mono}generators is affine.
\end{corollary}
\begin{proof}
Proposition \ref{prop:gen1_equations} shows that $S$ possesses an affine cover on which $\Gen_{1, S'/S}$ restricts to a single distinguished affine subset of the affine scheme $\WR = \Spec A[x_I]$.
\end{proof}

\begin{remark}
Consider a local index form $\localindex(e_1, \ldots, e_n) = \det(M(e_1, \ldots, e_n))$ defined by a basis $B \simeq \bigoplus_{i=0}^{n-1} A \cdot e_i$ as above. Suppose $\tilde{e}_1, \ldots, \tilde{e}_n$ is a second basis of $B$ over $A$ and $\tilde{M}$ is the matrix representation of $v^\sharp$ with respect to the bases $\{ 1, \ldots, t^{n-1} \}$
and $\{ \tilde{e}_1, \ldots, \tilde{e}_n \}$. Then $\det(\tilde{M}) = u\det(M)$ for some unit $u$ of $A$. Although the determinant of $M$ may not glue to a global datum on $S$, this shows the ideal that it generates does.
\end{remark}

\begin{remark}\label{rmk:nonmonhomogeneous}
The local index forms $\localindex(e_1, \ldots, e_n)$ are homogeneous with respect to the grading on $A[x_I]$.
To see this, note that since the $i$th column of $M(e_1, \ldots, e_n)$ represents $\theta^{i-1}$, its entries are of degree $i-1$ in $x_1, \ldots, x_n$. The
Leibnitz formula for the determinant 
\[\det(a_{ij}) = \sum_{\sigma \in \Sigma_n} (-1)^\sigma \cdot \prod a_{\sigma(i) i}\]
shows that $\localindex(e_1, \ldots, e_n)$ is homogeneous of degree $\sum_{i = 1}^n (i-1) = \frac{i(i-1)}{2}.$
The transition functions induced by change of basis respect this grading, so the index form ideal $\indexsheaf_{S'/S}$ is a sheaf of homogeneous ideals. 
\end{remark}

\begin{remark}\label{rmk:IndexFormEquations}
We compare our local index form to the classical number theoretic situation. For references see \cite{EvertseGyoryBook} and \cite{GaalsBook}. 
If $K$ is a number field and $L$ is an extension of finite degree $n$, then there are $n$ distinct embeddings of $L$ into an algebraic closure that fix $K$. Denote them $\sigma_1,\dots, \sigma_n$. Let $\Tr$ denote the trace from $L$ to $K$. The \textit{discriminant} of $L$ over $K$ is defined to be the ideal 
$\disc(L/K)$ generated by the set of elements of the form
\[
( \det[\sigma_i(\omega_j)]_{1\leq i,j\leq n})^2=\det[\Tr(\omega_i\omega_j)]_{1\leq i,j\leq n},\]
where we vary over all $K$-bases for $L$, $\{\omega_1,\dots, \omega_n \}$, with each $\omega_i\in \ZZ_K$.
If $\alpha$ is any element of $L$, then the \textit{discriminant} of $\alpha$ over $K$ is defined to be
\begin{align*}
    \disc_{L/K}(\alpha) &=\left( \det\left(\sigma_i\left(\alpha^{j-1}\right)\right)_{1\leq i,j\leq n}\right)^2     \\
    &=\det\left(\Tr\left(\alpha^{i-1}\alpha^{j-1}\right)\right)_{1\leq i,j\leq n}     \\
    &=\prod_{1\leq i<j\leq n}\left(\sigma_i(\alpha)-\sigma_j(\alpha)\right)^2,
\end{align*}
where the last equality comes from Vandermonde's identity. Note that $\disc_{L/K}(\alpha)$ is a power of the discriminant of the minimal polynomial of $\alpha$. 
For every $\alpha$ generating $L$ over $K$ one has
\[\disc_{L/K}(\alpha)=\left[\ZZ_L:\ZZ_K[\alpha]\right]^2\disc(L/K).\]
One defines the \textit{index form} of $\ZZ_L$ over $\ZZ_K$ be to
\[\classicalindex_{\ZZ_L/\ZZ_K}(\alpha) = \left[\ZZ_L:\ZZ_K[\alpha]\right]=\sqrt{\left|\frac{\disc_{L/K}(\alpha)}{\disc(L/K)}\right|}.\]
Confer \cite[Equation 5.2.2]{EvertseGyoryBook}.
In the case where $\{\omega_1,\dots, \omega_n \}$ is a $\ZZ_K$-basis for $\ZZ_L$, employing some linear algebra \cite[Equation (1.5.3)]{EvertseGyoryBook}, one finds $\classicalindex_{\ZZ_L/\ZZ_K}$ is, up to an element of $\ZZ_K^*$, the determinant of the change of basis matrix from $\{\omega_1,\dots, \omega_n \}$ to $\{1, \alpha,\dots, \alpha^{n-1} \}$.

The matrix in Equation \eqref{eq:MatrixOfCoeffs} is just such a matrix and its determinant coincides up to a unit with the index form in situations where the index form is typically defined. 
The generality of our setup affords us some flexibility that is not immediate from the definition of the classical index form equation.
\end{remark}

The local index forms give the complement of $\Gen_{S'/S}$ in $\WR_{S'/S}$ a closed subscheme structure.

\begin{definition}[Non-monogenerators $\NGen_{S'/S}$]\label{defn:nonmonogens}
Let $\indexsheaf_{S'/S}$ be the locally principal ideal sheaf on $\WR$ generated locally by local index forms. We call this the \emph{index form ideal}. Let $\NGen_{S'/S}$ be the closed subscheme of $\WR$ cut out by the vanishing of $\indexsheaf_{S'/S}$. We call this the \emph{scheme of non-monogenerators}, since its support is the complement of $\Gen_{S'/S}$ inside of $\WR$.
\end{definition}

\begin{proposition}
When the scheme of non-monogenerators $\NGen_{S'/S}$ is an effective Cartier divisor (equivalently, when none of the local index forms are zero divisors), the divisor class of $\NGen_{S'/S}$ in $\WR = \HHom_S(S', \AA^1)$ is the same as the pullback of the Steinitz class of $S'/S$ from $S$.
\end{proposition}
\begin{proof}
Recall that $V(m(t))$ is the vanishing of $m(t)$ in $\AA^1_{\WR}$, where $m(t)$ is the generic minimal polynomial for $S'/S$. Let $\tau$ be the natural map $\tau : V(m(t)) \to \WR$. Consider the morphism
\[
  v^\sharp : \tau_*\OO_{V(m(t))} \to \pi_*\OO_{S' \times_S \WR}
\]
of sheaves on $\WR$. The first sheaf is free of rank $n$ since $m(t)$ is a monic polynomial: there is a basis given by the images of $1, t, \ldots, t^{n-1}$. Therefore, taking $n$th wedge products in the previous equation, we have a map
\[
  \OO_{\WR}\cong \det(\tau_*\mathcal{O}_{V(m(t))}) \to \det\left(\pi_*\OO_{S' \times_S \WR}\right). 
\]
By construction, this map is locally given by a local index form, $\localindex(e_1, \ldots, e_n)$. Since we have assumed that $\NGen_{S'/S}$ is an effective Cartier divisor, the determinant is locally a nonzero-divisor. Therefore, the determinant identifies a non-zero section of $\pi_*\OO_{S' \times_S \WR}$. By definition of
$\NGen_{S'/S}$, we may identify $\det(\pi_*\OO_{S' \times_S \WR})$ with $\OO(\NGen_{S'/S})$.

Writing $\psi : \WR \to S$ for the structure map, we also have that
\[
  \det(\pi_*\OO_{S' \times_S \WR}) \cong \psi^*\det(\pi_*\mathscr{O}_{S'}).
\]
since taking a determinental line bundle commutes with arbitrary base change and $\pi_*$ commutes with base change for flat maps. The class of the line bundle $\det(\pi_*\mathscr{O}_{S'})$ in $\Pic(S)$ is by definition the Steinitz class.
\end{proof}

\subsection{Explicit equations for polygenerators $\Gen_k$} \label{ssec:genk_equations}
The work above readily generalizes to describe $\Gen_k$. 

Fix a number $k \in \NN$. We now construct explicit equations for $\Gen_k$ as a subscheme of $\WR_k = \HHom_S(S', \mathbb{A}^k)$ when $S' \to S$ is free and $S$ is affine. These hypotheses hold Zariski locally on $S$, so by Remark \ref{rmk:pbmonnormnofmon}, this gives a construction for $\Gen_k$ locally on $S$ in the general case.

Let $S' = \Spec B$ and $S = \Spec A$, where $B = \bigoplus_{i=1}^n A \cdot e_i$ is a finite free $A$-algebra of rank $n$ with basis $e_1, \ldots, e_n$. Let $J = \{ 1, \ldots, k \}$ and $I = \{ 1, \ldots, n \}$. Write $t_J$ for the $|J|$ coordinate functions $t_1, \ldots, t_k$ of $\mathbb{A}^k$ and write $x_{I \times J}$ as shorthand for $|I \times J|$ variables $x_{ij}$ indexed by $(i,j) \in I \times J$.

The scheme $\WR_k$ is represented by the affine scheme $\Spec A[x_{I \times J}]$ and the universal map for $\WR_k$ is the commutative triangle
\[\begin{tikzcd}
  S'\times_S \WR_k \ar[rr, "u"] \ar[dr]        &       & \AA^k_{\WR_k} \ar[dl]       \\
        &\WR_k
\end{tikzcd}\]
where the horizontal arrow $u$ is induced by the ring map $A[x_{I \times J}, t_J] \to B[x_{I \times J}]$ sending
\[
  t_j \mapsto \theta_j \coloneqq   \sum_{i \in I} x_{ij} e_i.
\]
Notice that $S' \times_S \WR_k \to \WR_k$ is in Situation \ref{sit:gensetup}. Apply Lemma \ref{lem:minpoly} to find monic degree $n$ polynomials $m_j(t_j) \in A[x_{I \times J}, t_J]$ such that
$m_j(\theta_j) = 0 $ in $B[x_{I \times J}]$.
Write $v^\sharp$ for the
unique map
\[
  v^\sharp : A[x_{I \times J}, t_J]/(m_j(t_j) : j \in J) \longrightarrow B[x_{I \times J}]
\]
factoring $u^\sharp : A[x_{I \times J}, t_J] \to B[x_{I \times J}]$.

Now, $A[x_{I \times J}, t_J]/(m_j(t_j) : j \in J)$ is a free $A[x_{I \times J}]$-module of rank $n^k$ with basis given by the equivalence classes of the products $t_1^{r_1}\cdots t_k^{r_k}$ as the powers $r_j$ vary between $0$ and $n - 1$. Since $B[x_{I \times J}]$
is also a free $A[x_{I \times J}]$-module with basis $e_1, \ldots, e_n$, we may choose an ordering of the powers $t_1^{r_1} \cdots t_k^{r_k}$ and represent the map $v^\sharp$ by an $n^k \times n$ matrix $M$.
For each subset $C \subseteq \{ 1, \ldots, kn \}$ of size $n$, let $M_C$ be the submatrix of $M$ whose columns are indexed by $C$ and let $\det(M_C)$ be the determinant.

\begin{proposition} \label{prop:genk_equations}
Suppose $S' \to S$ is finite free and $S$ is affine. Then with notation as above, $\Gen_k$ is the union of the distinguished affines $D(\det(M_C))$ inside $Y$.
\end{proposition}

\begin{proof}
Check on points as in Proposition \ref{prop:gen1_equations}. 
\end{proof}


\section{Examples of the scheme of monogenerators}\label{s:examples}

We conclude with several examples to illustrate the nature and variety of the scheme of monogenerators. We will consider situations in which the classical index form of Remark \ref{rmk:IndexFormEquations} is well-studied, such as field extensions and extensions of number rings, as well as the more exotic situation of jet spaces.
We will make frequent reference to computation of the index form using the techniques of Section \ref{ssec:gen1_equations}.

\subsection{First examples}\label{subsec: firstexamples}

\begin{example}[Quadratic Number Fields]\label{ex:general_quadratic}
Let $K = \QQ(\sqrt{d})$, for any square-free integer $d$. It is well-known that the ring of integers $\ZZ_K$ is monogenic: $\ZZ_L\cong \ZZ[\sqrt{d}]$ or $\ZZ[\frac{1 + \sqrt{d}}{2}]$, depending on $d \bmod 4$. We will confirm this using our framework, and determine the scheme $\Gen_1$ of monogenic generators.

Let $\alpha$ denote the known generator of $\mathcal{O}_L$, either $\sqrt{d}$ or $\frac{1 +\sqrt{d}}{2}$. Let us take $\{1, \alpha \}$ as the basis $e_1, \ldots, e_n$. The universal map diagram \eqref{eq:univ_factorization_mon} becomes:
\[
\begin{tikzcd}
\ZZ[a, b, \alpha]        &   \ZZ[a, b, t]/(m(t))  \ar[l]  &\ZZ[a, b, t]   \ar[l] \\
        &\ZZ[a, b] \ar[ul] \ar[ur] \ar[u]
\end{tikzcd}
\]
where the map $\ZZ[a,b,t]/(m(t)) \to \ZZ[a, b, \alpha]$ is given by $t \mapsto a + b\alpha$. The universal minimal polynomial $m(t)$ is given by $t^2 - \text{Tr}(a + b\alpha)t + \text{N}(a + b\alpha)$.

This diagram encapsulates all choices of generators as follows. The elements of $\ZZ[\alpha]$ are all of the form $a_0 + b_0\alpha$ for $a_0, b_0 \in \ZZ$. Integers $a_0, b_0 \in \ZZ$ are in bijection with maps $\phi : \ZZ[a,b] \to \ZZ$. Applying the functor $\ZZ \otimes_{\phi,\ZZ[a,b]} - $ to the diagram above yields a diagram
\[
\begin{tikzcd}
\ZZ[\alpha]        &   \ZZ[t]/(m(t))  \ar[l]  &\ZZ[t]   \ar[l] \\
        &\ZZ \ar[ul] \ar[ur] \ar[u]
\end{tikzcd}
\]
where the map $\ZZ[t]/(m(t)) \to \ZZ[\alpha]$ takes $t \mapsto a_0 + b_0\alpha$. The image is precisely $\ZZ[a_0 + b_0\alpha]$, and the index form that we are about to compute detects whether this is all of $\ZZ[\alpha]$.

 Returning to the universal situation, the matrix representation of the map $\mathbb{Z}[a,b,t]/(m(t))\to \mathbb{Z}[a,b,\alpha]$ (what we have been calling the matrix of coefficients \eqref{eq:MatrixOfCoeffs}) is given by
\[\begin{bmatrix}1 & a \\ 0 & b\end{bmatrix}.\]
Notice that we did not need to compute $m(t)$ to get this matrix. The determinant, $b$, is the local index form associated to the basis $\{ 1, \alpha \}$. Therefore $\Gen_1 \cong \ZZ[a, b, b^{-1}]$. Taking $\ZZ$-points of $\Gen_1$, we learn that $a + b\alpha$ ($a, b \in \ZZ$) is a monogenic generator precisely when $b$ is a unit, i.e. $b = \pm 1$.
\end{example}

Proposition \ref{prop:deg2samemonospace} generalizes this example to any degree-two $S' \to S$.

\begin{example} \label{ex:confspindexform}
Resuming the situation of Example \ref{ex:etale}, the local index form with respect to the basis $e_1,\ldots,e_n$ is the Vandermonde determinant:
\[
  \localindex(e_1,\ldots,e_n)(x_1,\ldots,x_n) = \pm \prod_{i < j} (x_i - x_j).
\]
Therefore $\Gen_{S'/S} = \Spec \CC[x_1,\ldots,x_n, (\prod_{i < j} (x_i - x_j))^{-1}]$. The claim of Example \ref{ex:confspsintro} follows.
\end{example}


\begin{example}[Jets in $\Aff^1$]
\label{ex:jets_in_A1}

Let $S = \Spec \ZZ$ and $S'_n = \Spec \ZZ[\e]/\e^n$, as in Example~\ref{ex:jetsps}. We explicitly describe $\Gen_{1, S'_n/S} \subseteq \WR = \jetsp{n-1, \AA^1}$.

Choose the basis $1, \e, \ldots, \e^{n-1}$ for $\ZZ[\e]/\e^n$.
With respect to this basis, we may write the universal map diagram as
\[
\begin{tikzcd}
\ZZ[x_1, \ldots, x_n, \e]/\e^n        &   \ZZ[x_1,\ldots,x_n,t]/(m(t))  \ar[l]  &\ZZ[x_1, \ldots, x_n, t]   \ar[l] \\
        &\ZZ[x_1,\ldots,x_n] \ar[ul] \ar[ur] \ar[u]
\end{tikzcd}
\]
where $t \mapsto x_1 + x_2\e + \cdots + x_n\e^{n-1}$.

Change coordinates by $t \mapsto t - x_1$ 
so that the image of $t$ is
\[
  t \mapsto \theta = x_2\e + \cdots + x_n\e^{n-1}.
\]
Update $m(t)$ accordingly: $m(t) = t^n$. Our next task is to compute the representation of $\theta^j$ in $\{1,\e, \ldots, \e^{n-1}\}$-coordinates for $j = 0, \ldots, n-1$. The multinomial theorem yields
\[
  (x_2 \e + x_3 \e^2 + \cdots + x_n \e^{n-1})^j = \sum_{i_2 + i_3 + \cdots + i_n = j} \dbinom{j}{i_2, \dots, i_n} \prod_{t = 2}^n x_t^{i_t} \e^{(t-1)i_t}.
\]

The coefficient of $\e^p$ is 
\[
  \sum_{\substack{i_2 + i_3 + \cdots + i_n = j \\
  i_2 + 2i_3 + \cdots (n-1) i_n = p}} \dbinom{j}{i_2, \dots, i_n} \prod_{t = 2}^n x_t^{i_t}.
\]

The matrix of coefficients in Figure \ref{fig:mjetmatrix} represents the $\ZZ[x_1,\ldots,x_n]$-linear map from $\ZZ[x_1,\ldots,x_n, t]/(m(t))$ to $\ZZ[x_1,\ldots,x_n,\e]/\e^{n} = \bigoplus_{i=0}^{n-1} \ZZ[x_1,\ldots,x_n] \cdot \e^i$. The coefficient of $\e^p$ above appears in the $(j+1)$st 
column and $(p+1)$st row. 

\begin{figure}[h]
\centering
\[M_n=\left[\begin{tikzcd}[column sep=tiny, row sep=tiny]
1       &0      &0      &0      &0      &\cdots     &0      \\
0       &x_2        &0       &0      &0      &\cdots        &0        \\
0       &x_3        &x_2^2       &0      &0      &\cdots         &0       \\
0       &x_4        &2x_2 x_3       &x_2^3      &0      &\cdots         &0       \\
0       &x_5        &2x_2 x_4 + x_3^2       &3x_2^2 x_3      &x_2^4      &\cdots         &0       \\
\vdots  &  \vdots   & \vdots                & \vdots         & \vdots    & \ddots        & \vdots   \\
0       &x_n        & \cdots                & \cdots         & \cdots    &\cdots         &x_2^{n-1}       
\end{tikzcd}\right]\]
\caption{The matrix determined by an $(n-1)$-jet.}\label{fig:mjetmatrix}
\end{figure}
Since $M_n$ is lower triangular, it has determinant $x_2^{\frac{n(n-1)}{2}}$. An $(n-1)$-jet thereby belongs to $\Gen_{\Aff^1, S'_n/S}$ if and only if the coefficient $x_2$ is a unit. The $x_i$ are naturally coordinates of the jet space, yielding $\Gen_{\Aff^1, S'_n/S} = \Spec \ZZ[x_1,\ldots,x_n,x_2^{-1}] \subseteq \WR_{S'_n/S} = J_{n-1, \Aff^1}$.

\end{example}

The scheme of $k$-generators $\Gen_k$ need not be affine. Even
for the Gaussian integers $\ZZ[i] / \ZZ$, we have that $\Gen_k = \Aff^k \times (\Aff^k \setminus \{ \vec{0} \})$. We prove this in Proposition \ref{prop:deg2samemonospace}, after a small lemma. 
The second factor begs to be quotiented by group actions of $\GG_m$, $\Sigma_n$, or $\GL_n$: doing so leads to the notion of twisted monogenicity considered in the sequel paper.

\begin{lemma}\label{lem:oneispartofbasis}
Locally on $S$, the ring $\OO_{S'}$ has an $\OO_S$-basis in which one basis element is $1$.
\end{lemma}

\begin{proof}
Omitted. 
\end{proof}

\begin{proposition}\label{prop:deg2samemonospace}
Suppose $S' \to S$ has degree 2 and let $\vct{0} \in \Aff^k_S$ be the zero section. Then affine locally on $S$ we have an isomorphism
\[
  \Gen_{k, S'/S} \cong \AA^k_S \times (\mathbb{A}^k_{S} \setminus \vct{0}).
\]
\end{proposition}
\begin{proof}
Working affine locally and applying Lemma \ref{lem:oneispartofbasis}, we may take $S' = \Spec B$ and $S = \Spec A$, where $B$
has an $A$-basis of the form $\{ 1, e \}$. Write $b_1, \dots, b_k$ for the coordinates on the second $\Aff^k_S$, so $\vct{0} = V(b_1, \dots, b_k) \subseteq \Aff^k_S$.

In the notation preceding Proposition \ref{prop:genk_equations}, we may
take $x_{i,1} = a_i$, $x_{i,2} = b_i$, and $t_i \mapsto a_i + b_ie$. The matrix of coefficients will have columns given by the $\{1, e\}$-basis representation of the images of $1, t_i$, and $t_it_j$ as $i,j$ vary over distinct integers in $\{1, \ldots, k\}$. Write $e^2 = c + de$ where $c, d \in A$. Then the column of the matrix representing $1$ is
\[
  \begin{bmatrix} 1 \\ 0 \end{bmatrix},
\]
the columns representing the images of the $t_i$ are
\[
  \begin{bmatrix} a_i \\ b_i \end{bmatrix},
\]
and the columns representing the images of $t_it_j$ are
\[
  \begin{bmatrix} a_ia_j + b_ib_jc \\ a_ib_j + a_jb_i + b_ib_jd \end{bmatrix}.
\]
Among the determinants of the $2\times2$ minors of this matrix are $b_1, \ldots, b_k$, coming from the submatrices
\[
  \begin{bmatrix} 1 & a_i \\ 0 & b_i \end{bmatrix}.
\]
The remaining determinants all lie in the ideal $(b_1,\ldots, b_k)$ since all elements of the second row of the matrix lie in this ideal.
We conclude by Proposition \ref{prop:genk_equations} that $\Gen_{k, S'/S}$ is the union of
the open subsets $D(b_i)$ of $\Spec A[a_1, b_1, \ldots, a_k, b_k]$, as required.
\end{proof}

As an alternative to taking the union of $k$-determinants, we can use a generalization of the determinant first introducted by Cayley, later rediscovered and generalized by Gel'fand, Kapranov and Zelevinsky: 

\begin{question}

The map $v^\sharp$ above is a multilinear map from the tensor product of $k$ free modules $A[x_{I \times J}, t_j]/m_j(t_j)$ of rank $n$ over $A[x_{I \times J}]$ to the rank-$n$ free module $B[x_{I \times J}]$. For $k=1$, $\Gen_1$ is the complement of the determinant of $v^\sharp$. In general, the map $v^\sharp$ is locally given by a \emph{hypermatrix} of format $(n-1, \dots, n-1)$ \cite{hyperdeterminant}. This $n \times n \times \cdots \times n$-hypercube of elements of $A[x_{I \times J}]$ describes a multilinear map the same way ordinary $n \times n$ matrices describe a linear map. What locus does the hyperdeterminant cut out in $\WR_k$?

\end{question}

Example \ref{ex:jetsofA2} addresses the case $k = 2$ for jet spaces.

\subsection{Field extensions}

When $S' = \Spec L \to S = \Spec K$ is induced by a field extension $L/K$, we know that the monogenic generators of $L$ over $K$
are precisely the elements of $L$ that do not belong to any proper subfield of $L$. Therefore, on the level of $K$-points of $\Gen_1$, we
can expect to see that the index form vanishes on precisely the proper subfields of $L$. However, it has further structure that
is better seen after extension to a larger field.

\begin{example}[A $\ZZ/2 \times \ZZ/2$ field extension]
Let $S = \Spec \QQ$ and $S' = \Spec \QQ(\sqrt{2}, \sqrt{3})$. 
The isomorphism of groups $\QQ(\sqrt{2}, \sqrt{3})\cong \QQ\oplus\QQ\sqrt{2}\oplus\QQ\sqrt{3}\oplus \QQ\sqrt{6}$ identifies the Weil Restriction $\WR_{\QQ(\sqrt{2}, \sqrt{3})/\QQ}$
and its universal maps with Spec of
\[\begin{tikzcd}
\QQ[a, b, c, d][\sqrt{2}, \sqrt{3}]         &       &\QQ[a, b, c, d][t]. \ar[ll, "a + b\sqrt{2} + c \sqrt{3} + d \sqrt{6} \mapsfrom t 
", swap]\\
        &\QQ[a, b, c, d] \ar[ur] \ar[ul]
\end{tikzcd}\]

Hence $\WR_{\QQ(\sqrt{2}, \sqrt{3})/\QQ} = \QQ[a, b, c, d]$ and the universal
morphism 
\[u : S' \times_S \WR_{\QQ(\sqrt{2}, \sqrt{3})/\QQ} \to \AA^1_{\WR_{\QQ(\sqrt{2}, \sqrt{3})/\QQ}}\] 
is induced by
\[
  t \mapsto a + b\sqrt{2} + c\sqrt{3} + d\sqrt{6}.
\]

We expand the images of the powers $1, t, t^2, t^3$ to find the matrix of coefficients
\[	
\begin{bmatrix}
1 & a & a^{2} + 2 b^{2} + 3 c^{2} + 6 d^{2} & a^{3} + 6 a
b^{2} + 9 a c^{2} + 36 b c d + 18 a d^{2} \\
0 & b & 2 a b + 6 c d & 3 a^{2} b + 2 b^{3} + 9 b c^{2} + 18
a c d + 18 b d^{2} \\
0 & c & 2 a c + 4 b d & 3 a^{2} c + 6 b^{2} c + 3 c^{3} + 12
a b d + 18 c d^{2} \\
0 & d & 2 b c + 2 a d & 6 a b c + 3 a^{2} d + 6 b^{2} d + 9
c^{2} d + 6 d^{3}
\end{bmatrix}.\]
We compute the local index form associated to our chosen basis by taking the determinant:
\begin{align*}
 \localindex(a,b,c,d) &= -8b^{4} c^{2} + 12 b^{2} c^{4} + 16 b^{4} d^{2} -36 c^{4} d^{2} -48 b^{2} d^{4} + 72 c^{2} d^{4} \\
  &= -4(2b^2 - 3c^2)(b^2 - 3d^2)(c^2 - 2d^2).
\end{align*}

Note that this determinant has degree 6. Dropping subscripts, the factorization implies that the closed subscheme of non-generators $\NGen$ inside $\WR \cong \AA^4_\QQ$ has three components of degree 2.

Consider the $\QQ$-points of $\Gen_{1,S'/S} = \WR - \NGen$.
These are in bijection with the elements $a + b\sqrt{2} + c\sqrt{3} + d\sqrt{6} \in \QQ(\sqrt{2}, \sqrt{3})$ where $a,b,c,d$ are in $\QQ$ and the index form does not vanish.
Equivalently,
\[
  2b^2 - 3c^2 \neq 0, \quad b^2 - 3d^2 \neq 0, \text{ and } \quad  c^2 - 2d^2 \neq 0.
\]

Let $\theta = a + b\sqrt{2} + c\sqrt{3} + d\sqrt{6}$ for some $a,b,c,d \in \QQ$ and consider what it would mean to fail one of these conditions. If
$2b^2 - 3c^2 = 0$ for $b, c \in \QQ$, it must be that $b = c = 0$. Then
$\theta \in \QQ(\sqrt{6})$, a proper subfield of $\QQ(\sqrt{2}, \sqrt{3})$.
Similarly, if $b^2 - 3d^2 = 0$ then $\theta \in \QQ(\sqrt{3})$, and
if $c^2 - 2d^2 = 0$ then $\theta \in \QQ(\sqrt{2})$. It follows that the
$\QQ$-points of $\Gen_{1, S'/S}$ are in bijection with the elements
of $\QQ(\sqrt{2}, \sqrt{3})$ that do not lie in a proper subfield, as
we expect from field theory.
\end{example}

See Example \ref{ex:NonMonOrderInMonogenicExt} for an analysis of the monogenicity of some orders contained in the field considered above.

\begin{example}[A $\ZZ/4\ZZ$-extension]
Let $S = \Spec \QQ(i)$ and $S' = \Spec \QQ(i, \sqrt[4]{2}).$ We have a global
$\QQ(i)$-basis $\{1, \sqrt[4]{2}, \sqrt{2}, \sqrt[4]{2}^3\}$ for $\QQ(i, \sqrt[4]{2})$ over $\QQ(i)$. We may use this basis to write $\WR \cong \Spec \QQ(i)[a,b,c,d]$ where the universal map
from $\mathbb{A}^1$ is $t \mapsto a+b\sqrt[4]{2}+c\sqrt{2}+d(\sqrt[4]{2})^3$. The matrix of coefficients is
\[
  \begin{bmatrix}
    1 & a & a^{2} + 2c^{2} + 4 b d & a^{3} + 6b^{2} c
+ 6 a c^{2} + 12a b d + 12 c d^{2} \\
    0 & b & 2 a b + 4 c d & 3 a^{2} b + 6 b c^{2} +
6 b^{2} d + 12 a c d + 4 d^{3} \\
    0 & c & b^{2} + 2 a c + 2 d^{2} & 3 a b^{2} + 3
a^{2} c + 2 c^{3} + 12 b c d + 6 a d^{2} \\
    0 & d & 2 b c + 2 a d & b^{3} + 6 a b c + 3 a^{2} d + 6 c^{2} d + 6 b d^{2}
  \end{bmatrix}.
\]

The determinant yields the local index form with respect to this basis:
\[
  (b^2 - 2d^2)(b^4 + 8c^4 - 16 b c^2 d + 4 b^2 d^2 + 4d^4).
\]

We note that the first factor vanishes for $a,b,c,d \in \QQ(i)$ when $a+b\sqrt[4]{2}+c\sqrt{2}+d\sqrt[4]{2}^3 \in \QQ(i, \sqrt{2})$. At first glance the second factor is more mysterious, but after adjoining enough elements, the entire index form factors into distinct linear terms:
\begin{align*}
 &(b - \sqrt{2}d)(b + \sqrt{2}d)(i b-(1+i) \sqrt[4]{2} c+\sqrt{2} d) (-i b-(1-i) \sqrt[4]{2} c+\sqrt{2}   d)\\ 
&\quad \quad \quad \cdot(-i b+(1-i) \sqrt[4]{2} c+\sqrt{2} d) (i b+(1+i) \sqrt[4]{2} c+\sqrt{2} d).
\end{align*}
\end{example}

This behavior of factorization into distinct linear factors occurs in general:

\begin{proposition}
Let $S' \to S$ be induced by a finite separable extension of fields $L/K$. Let $e_1, \ldots, e_n$ be a $K$-basis for $L$, and let $x_1, \ldots, x_n$ be the corresponding coordinates for $\WR$. Then the local index form $\localindex(e_1, \ldots, e_n)$ factors completely into distinct linear factors in $x_1, \ldots, x_n$ over the normal closure $\tilde{L}$ of $L/K$.
\end{proposition}

Compare this with Example \ref{ex:confspindexform}. There, monogenerators correspond to configurations of $n$ points in $\Aff^1$ and the distinct linear factors of $\localindex(e_1, \ldots, e_n)$ correspond to when pairs of points collide. Here the situation is the same except a separable field extension---geometrically, an \'etale localization---is required first. This \'etale local characterization of monogenerators is common to all \'etale $S' \to S$, a case that we will examine in more depth in the sequel paper.

Some interesting and useful specifics in the case of number fields are investigated in more depth in chapter 7 of \cite{GaalsBook}.

\begin{proof}
We may consider $\localindex(e_1, \ldots, e_n)$ as an element of $\tilde{L}[x_1, \ldots, x_n]$ by
pulling back to $\WR_{L/K} \times_S \Spec \tilde{L} \cong \WR_{L \otimes_K \tilde L/\tilde L}$.
Our strategy is to compute a second generator of the pullback of $\indexsheaf_{S'/S}$ with respect to a more convenient basis.

By the Chinese remainder theorem, $L \otimes_K \tilde{L} \cong \prod_{i = 1}^n \tilde{L}$. Let $\tilde{e}_1, \ldots, \tilde{e}_n$ be the
standard basis of $(\tilde{L})^n$, let $\tilde{x}_1, \ldots, \tilde{x}_n$ be the corresponding coordinates on $\WR_{L \otimes_K \tilde L/\tilde L} \simeq \WR_{\tilde L^n/\tilde L}$,
and let $\tilde{\theta} = \tilde{x}_1\tilde{e}_1 + \cdots + \tilde{x}_n\tilde{e}_n$.
Computing a matrix $M$ for the map $L[\tilde{x}_1, \ldots, \tilde{x}_n, t] / m(t) \to \tilde{L}[\tilde{x}_1, \ldots, \tilde{x}_n]$ sending $t \mapsto \tilde{\theta}$, we see that it is a Vandermonde matrix with factors
$\tilde{x}_1, \dots, \tilde{x}_n$, since
\[
  (\tilde{x}_1\tilde{e}_1 + \cdots + \tilde{x}_n\tilde{e}_n)^k = \tilde{x}_1^k\tilde{e}_1 + \cdots \tilde{x}_n^k\tilde{e}_n,
\]
when computed in the product ring $(\tilde{L})^n$. Therefore $|\localindex(\tilde{e}_1, \ldots, \tilde{e}_n)| = |\det(M)| = |\prod_{i < j} (\tilde{x}_i - \tilde{x}_j)|$. Applying the $\tilde{L}$-linear change of basis from
$\{\tilde{x}_i\}$ to $\{ x_i \}$, we see that $\localindex(e_1, \ldots, e_n)$ is a product of distinct linear factors in $x_1, \ldots, x_n$.
\end{proof}

The proposition above does not consider inseparable extensions. To see what can happen then, we begin with an example.

\begin{example}[A purely inseparable extension]\label{ex:inseparable} 
For $\mathbb{F}_3(\alpha)[\beta]/(\beta^3 - \alpha)$ over $\mathbb{F}_3(\alpha)$, write $a, b, c$ for the universal coefficients of the basis $1, \beta, \beta^2$. In other words, $\theta = a +b\beta+c\beta^2$. One computes that the 
index form is then 
\[b^3 - c^3\alpha.\] 
To find the monogenic generators of this extension, we look for $a, b, c \in \FF_3(\alpha)$ so that $b^3 - c^3\alpha \neq 0$. Clearly, at least one of $b, c$ must be nonzero. Choose $b, c$ arbitrarily so that one is nonzero. Is this enough to ensure we have a monogenerator?

Suppose first that $b \neq 0$. Then $b^3 - c^3\alpha = 0$
implies $c \in \FF_3(\alpha^{1/3}) \setminus \FF_3(\alpha)$, a contradiction. Symmetrically, if $c \neq 0$ and $b^3 - c^3\alpha = 0$, then $b \in \FF_3(\alpha^{1/3}) \setminus \FF_3(\alpha)$, a contradiction again. We conclude that the set of monogenerators is
\begin{align*}
  \Gen_{1,S'/S}(\FF_3(\alpha)) &= \{ a+b\beta + c\beta^2 \mid a,b,c\in\FF_3(\alpha)\text{ and }(b,c) \neq (0,0) \}. \\
    &= \FF_3(\alpha) \setminus \FF_3,
\end{align*}
as one expects from field theory.

The polynomial $b^3 - c^3\alpha$ is irreducible in $\FF_3(\alpha)[a,b,c]$, so the scheme of non-generators $\NGen$ is an irreducible subscheme of $\WR \simeq \AA^3$. However, $\NGen$ is not geometrically reduced: after base extension to $\FF_3(\alpha,\sqrt[3]{\alpha})$, the index form factors as $(b - c\beta)^3$.
\end{example}

The factorization noted above is not an isolated phenomenon:

\begin{proposition} \label{prop:pureinsepmon}
Let $S' \to S$ be induced by a degree $n\coloneqq p^m$ completely inseparable extension of fields $K(\alpha^{1/p^m})/K$. Then over $K(\alpha^{1/p^m})$, the local index form factors into a repeated linear factor of multiplicity $p^m$.
\end{proposition}

\begin{proof}
Consider the local index form $\localindex(e_1,\ldots, e_n)$ as an element of $K(\alpha^{1/p^m})[x_1, \ldots, x_n]$ by pulling back to
\[
  \WR_{K(\alpha^{1/p^m})/K} \times_S S' \cong \WR_{K(\alpha^{1/p^m}) \otimes_K K(\alpha^{1/p^m})/K(\alpha^{1/p^m})}.
\]
Once again, to arrive at the result, we will compute a second generator of the pull back of $\indexsheaf_{S'/S}$ with respect to another basis.

By the Chinese remainder theorem,
\begin{align*}
    K(\alpha^{1/p^m}) \otimes_K K(\alpha^{1/p^m}) &\cong K[t]/(t^{p^m} - \alpha) \otimes_K K(\alpha^{1/p^m}) \\
    &\cong K(\alpha^{1/p^m})[t]/((t - \alpha^{1/p^m})^{p^m}) \\
    &\cong K(\alpha^{1/p^m})[\epsilon]/\epsilon^{p^m},\\
\end{align*}
where $\epsilon = t - \alpha^{1/p^m}$. 

Let $b_1 = 1, b_2 = \epsilon, \ldots, b_{p^m} = \epsilon^{p^m - 1}$ be a basis for
$K(\alpha^{1/p^m})[\epsilon]/\epsilon^{p^m}$ over $K(\alpha^{1/p^m})$, and let $y_1, \ldots, y_n$ be corresponding coordinates on
$\WR_{K(\alpha^{1/p^m})/K} \times_S S'$. We are now in the situation of Example \ref{ex:jets_in_A1}. 
Following the calculation there, we do a second change of coordinates to the basis $c_1 = 1, c_2 = \epsilon - y_1, \ldots, c_n = (\epsilon - y_1)^{n-1}$ and let $z_1, \ldots, z_n$ be the corresponding coordinates on $\WR_{K(\alpha^{1/p^m})/K} \times_S S'$. Taking the determinant of the matrix $M$ of the map $K(\alpha^{1/p^m})[z_1, \ldots, z_n,t]/m(t) \to K(\alpha^{1/p^m})[\epsilon][z_1, \ldots, z_n]/\epsilon^{p^m}$ sending $t \mapsto z_1c_1 + \cdots + z_nc_n$
with respect to the bases $\{ 1, \ldots, t^{n-1} \}$ and $\{ c_1, \ldots, c_n \}$, we obtain
\[
  \localindex(c_1, \ldots, c_n) = z_2^{\left(\frac{p^m(p^m-1)}{2}\right)}.
\]
Applying the change of basis from $\{ z_1, \ldots, z_n \}$ to $\{ x_1, \ldots, x_n \}$, we see that $\localindex(e_1, \ldots, e_n)$
is a power of a linear term.
\end{proof}

\subsection{Orders in number rings}

\begin{example}[Dedekind's Non-Monogenic Cubic Field]\label{ex:Dedekind}
Let $\eta$ denote a root of the polynomial $X^3 - X^2 - 2X -8$ and consider the field extension $L\coloneqq  \QQ(\eta)$ over $K\coloneqq \QQ$. When Dedekind constructed this example \cite{Dedekind} it was the first example of a non-monogenic extension of number rings. Indeed two generators are necessary to generate $\Ints_L/\Ints_K$: take $\eta^2$ and $\frac{\eta + \eta^2}{2}$, for example. In fact, $\{1,\frac{\eta + \eta^2}{2},\eta^2\}$ is a $\ZZ$-basis for $\Ints_K$. 
The matrix of coefficients with respect to the basis $\{1,\frac{\eta + \eta^2}{2},\eta^2\}$ is

\[\begin{bmatrix}
1 & a & a^{2} + 6b^{2} + 16bc + 8 c^{2} \\
0 & b & 2 a b + 7 b^{2} + 24bc + 20c^{2} \\
0 & c & -2b^{2} + 2ac - 8bc - 7c^{2}
\end{bmatrix}.\]

Taking its determinant, the index form associated to this basis is \[
 -2b^3 - 15b^2c - 31bc^2 - 20c^3
.\]
Were the extension monogenic, we would be able to find $a,b,c \in \ZZ$ so
that the index form above is equal to $\pm 1$.

To see that there are no solutions, we may reduce the index form modulo 2 to obtain
\[
  b^2c + bc^2.
\]
Iterating through the four possible values of $(b,c) \in (\ZZ/2\ZZ)^2$ shows that the index form always to reduces to 0.
\end{example}

\begin{example}[A non-monogenic order and monogenic maximal order]\label{ex:NonMonOrderInMonogenicExt}

Consider the extension $\ZZ[\sqrt{2}, \sqrt{3}]$ of $\ZZ$. Note that $\ZZ[\sqrt{2},\sqrt{3}]$ is not the maximal order of $\QQ(\sqrt{2},\sqrt{3})$. As we will see below, the maximal order is $\ZZ[\sqrt{\sqrt{3}+2}]$. The isomorphism of groups $\ZZ[\sqrt{2}, \sqrt{3}] \simeq \ZZ \oplus \ZZ \sqrt{2} \oplus \ZZ \sqrt{3} \oplus \ZZ \sqrt{6}$ identifies the Weil Restriction $\WR_{\ZZ[\sqrt{2}, \sqrt{3}]/\ZZ}$
and its universal maps with Spec of
\[\begin{tikzcd}
\ZZ[a, b, c, d][\sqrt{2}, \sqrt{3}]         &       &\ZZ[a, b, c, d][t]. \ar[ll, "a + b\sqrt{2} + c \sqrt{3} + d \sqrt{6} \mapsfrom t 
", swap]\\
        &\ZZ[a, b, c, d] \ar[ur] \ar[ul]
\end{tikzcd}\]

Now, 
\[1\mapsto 1\]
\[t\mapsto a + b\sqrt{2} + c\sqrt{3} + d\sqrt{6}\]
\[t^2\mapsto (a + b\sqrt{2} + c\sqrt{3} + d\sqrt{6})^2\]
\[t^3\mapsto (a + b\sqrt{2} + c\sqrt{3} + d\sqrt{6})^3\]

is given by
\[	
\begin{bmatrix}
1 & a & a^{2} + 2 b^{2} + 3 c^{2} + 6 d^{2} & a^{3} + 6 a
b^{2} + 9 a c^{2} + 36 b c d + 18 a d^{2} \\
0 & b & 2 a b + 6 c d & 3 a^{2} b + 2 b^{3} + 9 b c^{2} + 18
a c d + 18 b d^{2} \\
0 & c & 2 a c + 4 b d & 3 a^{2} c + 6 b^{2} c + 3 c^{3} + 12
a b d + 18 c d^{2} \\
0 & d & 2 b c + 2 a d & 6 a b c + 3 a^{2} d + 6 b^{2} d + 9
c^{2} d + 6 d^{3}
\end{bmatrix}.\]
Taking the determinant, the index form with respect to our chosen basis is
\[-8b^{4} c^{2} + 12 b^{2} c^{4} + 16 b^{4} d^{2} -36 c^{4} d^{2} -48 b^{2} d^{4} + 72 c^{2}
d^{4}\]
\[	
=-4(2b^2 - 3c^2)(b^2 - 3d^2)(c^2 - 2d^2).
\]

The $\ZZ$-points of $\Gen_{1,\ZZ[\sqrt{2}, \sqrt{3}]/\ZZ}$ are
in bijection with the tuples $(a, b, c, d) \in \ZZ^4$ such that the determinant is a unit. Since the determinant is divisible
by 2, this never happens. We conclude that $\ZZ[\sqrt{2}, \sqrt{3}]$ is not monogenic over $\ZZ$.

The non-monogenicity of the order $\ZZ[\sqrt{2},\sqrt{3}]$ is in marked contrast to the maximal order of $\QQ(\sqrt{2},\sqrt{3})$, which is monogenic. 
A computation shows that a power integral basis for the maximal order is given by $\{1,\alpha,\alpha^2,\alpha^3\},$ where $\alpha$ is a root of $t^4 - 4t^2 + 1$. One could take $\alpha=\sqrt{\sqrt{3}+2}$. Here the Weil Restriction $\WR_{\ZZ[\alpha]/\ZZ}$ 
and its universal maps are identified with Spec of
\[\begin{tikzcd}
\ZZ[a, b, c, d][\alpha]         &       &\ZZ[a, b, c, d][t]. \ar[ll, "a + b\alpha + c \alpha^2 + d \alpha^3 \mapsfrom t 
", swap]\\
        &\ZZ[a, b, c, d] \ar[ur] \ar[ul]
\end{tikzcd}\]

The element-wise computation  
\[1\mapsto 1\]
\[t\mapsto a + b\alpha + c \alpha^2 + d \alpha^3\]
\[t^2\mapsto (a + b\alpha + c \alpha^2 + d \alpha^3)^2\]
\[t^3\mapsto (a + b\alpha + c \alpha^2 + d \alpha^3)^3\]

yields the matrix of coefficients
\[\begin{bmatrix}
1 & a & a^2 - c^2 - 2bd - 4d^2 
& A \\
0 & b & 2ab - 2cd 
& B \\
0 & c & b^{2} + 2 a c + 4 c^{2} + 8 b d + 15 d^{2} 
&  C \\
0 & d & 2 b c + 2 a d + 8 c d 
& D
\end{bmatrix},\]
where
\begin{align*}
    A &= a^{3} - 3 b^{2} c - 3 a c^{2} - 4 c^{3} - 6 a b d - 24 b c d - 12 a d^{2} - 45 c d^{2},     \\
    B &= 3 a^{2} b - 3 b c^{2} - 3 b^{2} d - 6 a c d - 12 c^{2} d - 12 b d^{2} - 15 d^{3},        \\
    C &= 3 a b^{2} + 3 a^{2} c + 12 b^{2} c + 12 a c^{2} + 15 c^{3} + 24 a b d + 90 b c d + 45 a d^{2} + 168 c d^{2},         \\
    D &= b^{3} + 6 a b c + 12 b c^{2} + 3 a^{2} d + 12 b^{2} d + 24 a c d + 45 c^{2} d + 45 b d^{2} + 56 d^{3}.
\end{align*}
The determinant of this matrix yields the index form
\[(b^{2} - 2c^{2} + 6bd + 9d^{2}) (b^{2} -6  c^{2} + 10  b d + 25  d^{2}) (b^{2} + 4 b d +d^{2}).\]

One can compute that the index of $\ZZ[\sqrt{2}, \sqrt{3}]$ inside of $\ZZ[\sqrt{\sqrt{3} + 2}]$ is 2. Therefore the index forms are equivalent away from the prime 2.
\end{example}

\begin{example}\label{ex:2genicOverZ}
Let $K = \QQ$, $L = K(\sqrt[3]{5^2\cdot7})$. The ring of integers $\Ints_L = \ZZ[\sqrt[3]{5^2\cdot 7}, \sqrt[3]{5\cdot 7^2}]$ is not monogenic over $\ZZ$. Let $\alpha= \sqrt[3]{5^2\cdot 7}$, $\beta = \sqrt[3]{5\cdot 7^2}$. It turns out that $\{1, \alpha, \beta\}$ is a $\ZZ$-basis for $\ZZ_L$, so the universal map may be identified with
\[\begin{tikzcd}
\ZZ_L[a, b, c]         &       &\ZZ[a, b, c][t]. \ar[ll, "a + b\alpha + c \beta \mapsfrom t", swap]\\
        &\ZZ[a, b, c] \ar[ur] \ar[ul]
\end{tikzcd}\]

Expanding
\[1\mapsto 1\]
\[t\mapsto a + b\alpha + c\beta,\]
\[t^2\mapsto (a + b\alpha + c\beta)^2\]
we find that the matrix of coefficients is
\[\begin{bmatrix}
1 & a & a^2 + 70bc \\
0 & b & 2ab + 7c^2 \\
0 & c & 2ac + 5b^2
\end{bmatrix}.\]

Computing the determinant, we get the index form $5b^3 - 7c^3$. Reducing modulo 7, we see that the index form cannot be equal to $\pm 1$, so the extension is not monogenic.

\end{example}

\subsection{Other examples}

\begin{example}\label{ex:functionfieldintegers}
We investigate the analog of the integers in Example \ref{ex:inseparable}. We keep the same notation. The base ring is $\FF_3[\alpha]$ and the extension ring is $\FF_3[\alpha][x]/(x^3 - \alpha)=\FF_3[\beta]$, where $\beta^3=\alpha$.
\[\begin{tikzcd}
\FF_3[a, b, c][\beta]         &       &\FF_3[\alpha][a, b, c][x]. \ar[ll, "a + b\beta + c \beta^2 \mapsfrom x 
", swap]\\
        &\FF_3[\alpha][a, b, c] \ar[ur] \ar[ul]
\end{tikzcd}\]
\[1\mapsto 1\]
\[x\mapsto a + b\beta + c\beta^2\]
\[x^2\mapsto (a + b\beta + c\beta^2)^2\]

is given by
\[	
\begin{bmatrix}
1 & a & a^2+2bc\alpha \\
0 & b & c^2\alpha  +2 a b  \\
0 & c & b^2+2 a c  \\
\end{bmatrix}.\]
The determinant is $b^3 - c^3\alpha$, which is not geometrically reduced: it factors as $(b - c\beta)^3$. To find the monogenerators of this extension, we set this expression equal to the units of $\FF_3[\alpha]$. Since $(\FF_3[\alpha])^*=\pm 1$, the only solutions are $b = \pm1,$ $c=0$. Thus
\[\Gen_{1,\FF_3[\beta] / \FF_3[\alpha]} (\FF_3[\alpha]) = \{a\pm \beta : a\in\FF_3[\alpha]\}.\]
We can see that, much like number rings, monogenicity imposes a stronger restriction here than it does for the extension of fraction fields.
\end{example}


\begin{example}[Jet spaces of $\Aff^2$]\label{ex:jetsofA2}
Consider an $m$-jet of $\Aff^2 = \Spec k[t, u]$ determined as in Example \ref{ex:jets_in_A1} by
\[t = a_0 + a_1 \e + a_2 \e^2 + \cdots a_m \e^m\]
\[u = b_0 + b_1 \e + b_2 \e^2 + \cdots b_m \e^m.\]
Linear changes of coordinates ensure $a_0 = b_0 = 0$ and that our jets satisfy $t^{m +1} = u^{m+1} = 0$ in $k[\e]/\e^{m+1}$. To find the matrix for the induced $k$-linear map from 
\[k[t, u]/(t^{m+1}, u^{m+1}) = \bigoplus k \cdot t^e u^f\] 
to the jets $k[\e]/\e^{m+1} = \bigoplus k \cdot \e^i$, we need the coefficient of each $\e^p$ in the expression:
\begin{align*}
    t^e u^f &= (a_1 \e + a_2 \e^2 + \cdots + a_m \e^m)^e (b_1 \e + b_2 \e^2 + \cdots + b_m \e^m)^f       \\
    &=\left(\sum_{1 \leq r \leq m} \e^r \cdot \sum_{\substack{i_1 + i_2 + \cdots + i_m = m \\
    i_1 + 2i_2 + \cdots m i_m = r}} \dbinom{e}{i_1, \dots, i_m} \prod_{t=0}^m a_t^{i_t}\right) \\
    & \quad \quad \quad \quad \quad \quad \quad \quad \cdot \left(\sum_{1 \leq s \leq m} \e^s \cdot \sum_{\substack{j_1 + j_2 + \cdots + j_m = m \\
    j_1 + 2j_2 + \cdots m j_m = s}} \dbinom{f}{j_1, \dots, j_m} \prod_{t=0}^m b_t^{j_t}\right)      \\
    &=\sum_{1 \leq p \leq m} \e^p \left(
    \sum_{\substack{i_1 + i_2 + \cdots + i_m = m \\
    j_1 + j_2 + \cdots + j_m = m \\
    (i_1 + j_1) + 2(i_2 + j_2) + \cdots + m (i_m + j_m) = p}} \hspace{-.7 cm}\dbinom{e}{i_1, \dots, i_m} \dbinom{f}{j_1, \dots, j_m} \prod_{t=0}^m a_t^{i_t} b_t^{j_t}
    \right)
\end{align*}
If $e + f > p$, the coefficient of $\e^p$ in $t^e u^f$ is again zero. If $e + f > m$, all the coefficients are zero. The corresponding $m^2 \times m$ matrix is ``lower triangular'' in this sense. 



Take $m=1$ to reduce to A. Cayley's original situation of a $2 \times 2 \times 2$ hypermatrix; compute his second hyperdeterminant $\text{Det}$ to be $a_1^2 b_1^2$. In this case, $\WR_2$ is the tangent space of $\Aff^2$, the index forms cut out the locus where both $a_1$ and $b_1$ are zero, and the hyperdeterminant cuts out the locus where \textit{either} $a_1$ or $b_1$ are zero. 

Computability is a serious constraint for even simple cases. Taking $\Aff^3$ and $m=1$ yields a $2 \times 2 \times 2 \times 2$ hypermatrix. The formula for such a hyperdeterminant is degree 24 and has 2,894,276 terms \cite[Remark 5.7]{hyperdeterminant}. 

\end{example}

\begin{example}[Limits and Colimits]

Let $B$ be an $A$-algebra which is complete with respect to $I \subseteq B$. If each $B_m \coloneqq   B/I^m$ is finite locally free over $A$ and $X \to \Spec A$ is quasiprojective, there are \emph{affine} restriction maps $\Gen_{X, B_{m+1}/A} \to \Gen_{X, B_m/A}$. One can define
\[\Gen_{X, B/A} \coloneqq   \lim_m \Gen_{X, B_m/A},\]
which is a scheme \cite[01YX]{sta}. By \cite[Remark 4.6, Theorem 4.1]{bhatttannakaalgn}, this limit parametrizes closed embeddings $s : \Spec B \to X$ over $\Spec A$ as in Definition \ref{def:defnofmonogenweilrestn}. The arc space examples $k \adj{t}/k$, $k \adj{x, y}/k$ were mentioned in Example \ref{ex:jetsps}.

We cannot make a similar statement for colimits of algebras. Suppose $\{ B_i \}$ is a diagram of $A$-algebras indexed by $\NN$. Then for each $i < j$ there is a natural map $\WR_{B_i/A} \to \WR_{B_j/A}$. Notice that if the image of some $\theta \in B_i$ is a monogenerator of $B_j$, then $B_i \to B_j$ is surjective. It follows that $\WR_{B_i/A} \to \WR_{B_j/A}$ only takes $\Gen_{B_i/A}$ into $\Gen_{B_j/A}$ if $\Spec B_j \to \Spec B_i$ is a closed immersion over each open set $U \subseteq \Spec A$ over which $\Gen_{B_i/A}$ is non-empty. Assuming $\Gen_{B_0/A}$ is locally non-empty, the only diagrams $\{ B_i \}$ for which the colimit $\colim_i \Gen_{B_i/A}$ can even
be formed are those for which each $B_i \to B_j$ is surjective. Since $B_0$ is Noetherian, all such diagrams are eventually constant and uninteresting.

\end{example}


\section{Finite flat algebras with monogenerators}\label{ss:finflat+gen}

We mention a related moduli problem and how it fits into the present schema. We rely on a classical representability result:

\begin{theorem}[{\cite[Theorem 5.23]{fgaexplained}}]\label{thm:homrepable}
If $X \to S$ is flat and projective and $Y \to S$ quasiprojective over a locally noetherian base $S$, the functor $\HHom_S(X, Y)$ is representable by an $S$-scheme. 
\end{theorem}

The scheme $\HHom_S(X, Y)$ is a potentially infinite disjoint union of quasiprojective $S$-schemes. 

Fix a flat, projective map $C \to S$ and quasiprojective $X \to S$. Assign to any $S$-scheme $T$ the groupoid of finite flat maps $Y \to C \times_S T$ of degree $n$. One can think of this as a $T$-indexed family of finite flat maps $Y_t \to C$. This problem is represented by
\[\WR_{\frak A_n, C/S} \coloneqq \HHom_S(C, \frak A_n).\]
We study moduli of finite flat maps $Y \to C \times_S T$ \emph{together with} a choice of monogenerator:

\begin{definition}\label{def:relgenvariant}
The moduli problem $\scr F$ on $S$-schemes $\Sch{S}$ has $T$-points given by:
\begin{itemize}
    \item A finite, flat family $Y \to C \times_S T$ of degree $n$,
    \item A closed embedding $Y \subseteq X \times_S C \times_S T$ over $C \times_S T$.
\end{itemize}
These data form a fibered category via pullback. Define a variant $\scr F'$ parameterizing the data above together with a global basis $\cal Q \simeq \OO_T^{\oplus n}$ for the finite, flat algebra $\cal Q$ corresponding to $Y \to C \times_S T$.  
\end{definition}

The map $\scr F' \to \scr F$ forgetting the basis is a torsor for the smooth group scheme $\HHom_S(C, \GL_n)$. Let $X_C$ denote the pullback $X \times_S C$. The stack $\scr F$ is the Weil Restriction $\WR_{\Hilb_{X_C/C}^n, C/S}$ of the Hilbert Scheme $\Hilb_{X_C/C}^n$ for $X_C \to C$ along the map $C \to S$. Both are therefore representable by \emph{schemes} using the theorem. One must use caution: $\Hilb_{X_C/C}^n$ is an infinite disjoint union of projective schemes indexed by Hilbert polynomials and not itself projective, but this suffices for representability.

There are universal finite flat maps
\[\begin{tikzcd}
\tilde Z \ar[d]        &       &\tilde Y \ar[d]       \\
C \times_S \HHom_S(C, \frak B_n)         &       &C \times_S \HHom_S(C, \frak A_n),
\end{tikzcd}\]
with and without a global basis $\cal Q \simeq \OO_T^{\oplus n}$. The sheaf $\scr F$ may also be obtained by the Weil Restriction along $C \times_S \HHom_S(C, \frak A_n) \to \HHom_S(C, \frak A_n)$ of the monogenicity space $\Gen_{X, \tilde Y/C \times_S \HHom_S(C, \frak A_n)}$. The same construction of $\scr F'$ can be obtained with $\frak B_n$ in place of $\frak A_n$. 

We argue $\HHom_S(C, \frak A_n)$ is also representable by an algebraic stack. Olsson's result \cite[Theorem 1.1]{olssonhomstack} does not apply here because $\frak A_n$ is not separated. This means the diagonal $\Delta_{\frak A_n}$ is not proper, and this diagonal is a pseudotorsor for automorphisms of the universal finite flat algebra. The automorphism sheaf $\AAut(\cal Q)$ of some finite flat algebras is not proper: take the 2-adic integers $\ZZ_{2}$, $\cal Q = \ZZ_{2}[x]/x^2$, and the map $\cal Q \to \cal Q$ sending $x \mapsto 2x$. This is an automorphism over the generic point $\QQ_2 = \ZZ_{2}[\frac{1}{2}]$ and the zero map over the special point $\ZZ/2\ZZ = \ZZ_{2}/2 \ZZ_2$.

The scheme $\frak B_n$ on the other hand is a closed subscheme of an affine space, hence separated. The Weil Restriction $\HHom_S(C, \frak B_n)$ is a scheme by the above theorem and the map 
\[\HHom_S(C, \frak B_n) \to \HHom_S(C, \frak A_n)\]
is again a torsor for the smooth group scheme $\HHom_S(C, \GL_n)$. Therefore $\HHom_S(C, \frak A_n)$ is algebraic. 

The diagonal $\Delta_{\frak A_n}$ even fails to be quasifinite because some finite flat algebras have infinitely many automorphisms:

\begin{example}[Infinite automorphisms]
The dual numbers $k[\e]/\e^n$ have an action of $\GG_m$ by $\e \mapsto u\cdot \e$ for a unit $u \in \GG_m(k)$. 

For another example, let $k$ be an infinite field of characteristic three and consider $\cal Q = k[x, y]/(x^3, y^3-1)$. Because $(y + x)^3 = y^3$, there are automorphisms $y \mapsto y + u x$ for any $u \in k$.

\end{example}

The reader may define \emph{stable algebras} $\cal Q$ as those with unramified automorphism group \cite[0DSN]{sta}. There is a universal open, Deligne-Mumford substack $\tilde{\frak A}_n \subseteq \frak A_n$ of stable algebras \cite[0DSL]{sta}. This locus consists of points where the action $\GL_n \action \frak B_n$ has unramified stabilizers \cite[\S 2]{poonenmodspoffiniteflat}.


\printbibliography

\end{document}